\documentclass{amsart}
\usepackage{amssymb,amsmath,amsthm,latexsym}
\usepackage{amsfonts}
\usepackage{graphicx}
\usepackage{hyperref}
\usepackage{tikz}
\usepackage{tikz-cd}
\usepackage[hypcap=false]{caption}
\usepackage[margin=1in]{geometry}
\usepackage{enumitem}
\usepackage[T1]{fontenc}
\usepackage{mathrsfs}  
\usepackage[capitalise]{cleveref}

\usepackage[
backend=biber,
style=alphabetic,
sorting=nyt
]{biblatex}
\usepackage{booktabs}
\usepackage{tabularx}
\newtheorem{theorem}{Theorem}[section]

\newtheorem{corollary}[theorem] {Corollary}
\newtheorem{definition}[theorem]{Definition}

\newtheorem{lemma}[theorem]{Lemma}

\newtheorem{proposition}[theorem]{Proposition}
\newtheorem{remark}[theorem]{Remark}

\vspace{5cm}
\begin{document}
  
  \label{'ubf'}  
\setcounter{page}{1}                                 

\markboth {\hspace*{-9mm} \centerline{\footnotesize \sc
  Conductors and Quadratic base change  }
                 }
                { \centerline                           {\footnotesize \sc  
                   Siddharth Ramakrishnan Cherukara
                                                       } \hspace*{-9mm}              
               }

\vspace*{-2cm}

\begin{center}
{ 
       {\Large \textbf { \sc 
    Conductors and Quadratic base change}
       }
\\

\medskip

{\sc Siddharth Ramakrishnan Cherukara }\\
{\footnotesize Department of Mathematics}\\
{\footnotesize University of Oklahoma
}\\
{\footnotesize e-mail: {\it siddharth.ramakrishnan@ou.edu}}
}
\end{center}
\thispagestyle{empty}
\hrulefill
\thispagestyle{empty}
\hrulefill
\begingroup
\makeatletter
  \normalfont\Small
  \list{}{\labelwidth\z@
    \leftmargin3pc \rightmargin\leftmargin
    \listparindent\normalparindent \itemindent\z@
    \topsep\z@ \partopsep\z@
    \parsep\z@ \@plus\p@}%
  \item[\hskip\labelsep\scshape Abstract.]%
\makeatother
Let $F$ be a totally real field and $K/F$ a totally real quadratic extension. For a Hilbert cusp newform $f$ over $F$ with base change $f_K$ to $K,$ we determine the exact level of $f_K$ in terms of the level of $f,$ using only local representation theory: for each place $\mathfrak p$ of $F$ we compute how the conductor of the local component of $f$ changes under quadratic base change. The tame case (odd residue characteristic) follows from existing base change theory, but the case $p=2$ requires a finer $\varepsilon$-factor analysis that does not appear in the literature. We give a complete treatment of the dyadic imprimitive representations and, more delicately, the exceptional supercuspidal representations, for which no closed-form base change formula was previously known; in particular we determine, in terms of local root numbers, the base change of every exceptional supercuspidal representation of $GL(2,\mathbb Q_2)$ with trivial central character and conductor $3.$ This yields an explicit formula for the level of $f_K,$ specialized here to $F=\mathbb Q,$ along with a partial converse.
\endlist
\endgroup
\section{Introduction}
Let $S^{\text{new}}_{k}(N,\chi)$ denote the space of weight $k,$ level $N$ cusp forms with nebentypus $\chi.$ In \cite{MR253990}, Doi and Naganuma described, for given $f \in S^\text{new}_{k}(N),$ how to associate a \emph{Hilbert modular form} $F_{f}$ over $K,$  where $K$ is a totally real quadratic extension of $\mathbb{Q}.$ The Doi--Naganuma lifting was extended by Saito in \cite{MR406936} to the case where $K$ is a totally real cyclic extension of prime degree over $\mathbb{Q}$ satisfying some assumptions. This was seen as an instance of Langlands' functoriality principles. Let $F$ be a totally real number field, and let $K$ be an extension of $F,$ and let $\mathfrak{p}$ and $\mathfrak{P}$ be finite places of $F$ and $K$ respectively such that $\mathfrak{P}$ lies over $\mathfrak{p}.$ Locally, base change is a map from irreducible admissible representations of $GL(n,F_\mathfrak{p})$ to irreducible admissible representations of $GL(n,K_{\mathfrak{P}}),$ and globally is a map from cuspidal automorphic representations of $GL(n,\mathbb{A}_F)$ to those of $GL(n,\mathbb{A}_K).$The complete description  for $GL(2)$ was given in \cite{MR574808}, following \cite{MR406936}, using a local interpretation due to \cite{MR546611}. The case of $GL(n)$ was established in \cite{MR1007299}. Both works dealt with the cases where $K/F$ is a Galois extension of prime degree. In \cite{MR1685898} and \cite{MR2130587}, the authors extended the lift to cover the cases where $K,F$ are local fields and $K/F$ is any tamely ramified extension, and the notions agree with \cite{MR1007299} when the domains overlap. Our main problem is as follows.

Let $F$ be a totally real field,$\mathfrak{o}_F$ be its ring of integers and let $N$ be an ideal in $\mathfrak{o}_F.$Let $K$ be a totally real quadratic extension of $F.$

\paragraph{Let $f$ be a cuspidal newform over $F$ of weight $\textbf{k}$ and level $N.$ For a prime ideal $\mathfrak{p}$in $\mathfrak{o}_F$ not dividing $N,$ let $a_\mathfrak{p},b_\mathfrak{p}$ denote the Satake parameters of $f$at $\mathfrak{p}.$  We say that a modular form $g$ over $K$ of weight $\underline{\textbf{k}}=(\textbf{k,k})$ is a base change from $f$if, for all but finitely many $\mathfrak{p}$ and $\mathfrak{P}\mid \mathfrak{p},$ the Satake parameters $a_\mathfrak{P},b_\mathfrak{P}$ are $f(\mathfrak{P}/\mathfrak{p})\text{-th}$ powers of $a_\mathfrak{p},b_\mathfrak{p},$ where $f(\mathfrak{P}/\mathfrak{p})=[\mathfrak{o}_K/\mathfrak{P}:\mathfrak{o}_F/\mathfrak{p}]$   
} 
\label[definition]{def:base change}

\subsection*{Main Problem} Let $K/\mathbb{Q}$ be a quadratic extension. Let $N_{K/\mathbb{Q}}:K^\times \rightarrow \mathbb{Q}^\times$ denote the norm map. For every $f\in S^\text{new}_{\textbf{k}}(N,\chi)$ we can associate a newform $f_K$ in $S_{\underline{\textbf{k}}}^\text{new}(\mathfrak{N},\chi_{K}),$ the space of weight $\underline{\textbf{k}}$  Hilbert modular forms over $K$ for a suitable level $\mathfrak{N},$ where $\chi_K=\chi\circ N_{K/\mathbb{Q}}$ as in the above definition. How does the level $\mathfrak{N}$ depend on $f$?
\subsection{Method of proof} We approach this problem using only local representation theory. Let $F_\mathfrak{p}$ be a $p$-adic field. To each irreducible admissible representation $\pi$ of $GL(2,F_\mathfrak{p}),$ we associate a non-negative integer $c(\pi)$ called the \emph{conductor of} $\pi$ by \cite{MR337789}. We investigate how the conductor of a representation changes under local base change by a quadratic extension. By the work of Langlands, to each $f$ in $S^\text{new}_{k}(N)$ we can associate a representation $\bigotimes \pi_{p},$ where $\pi_{p}$ is an irreducible admissible representation of $GL(2,\mathbb{Q}_{p}),$ such that $N=\prod_{p} p^{c(\pi_{p})}.$ Thus, the local question completely resolves the global problem. The behaviour of the conductor under base change is completely understood in the case of tamely ramified extensions in \cite{MR1685898} and easily follows from \cite{MR1007299}. However the wildly ramified case is more delicate and requires finer analysis. Since our focus is on applications to modular forms, we will do a finer analysis in the case $p=2.$  This is difficult for several reasons. \begin{itemize}
    \item We can study the conductors via $\varepsilon\text{-factors}$ attached to these representations. These are not easy to describe when $p=2.$
    \item There are \emph{exceptional supercuspidal} representations in this setting, whose base changes do not admit descriptions as simple as those of the other cases.
    \item There are $2^{[F_\mathfrak{p}:\mathbb{Q}_{2}]+2}-1$ quadratic extensions of $F_\mathfrak{p}$ when $\mathfrak{p}$ is dyadic, giving more cases and calculations to consider.
\end{itemize}
These cases are dealt with in Sections 4 and 5, which form the heart of the paper. Although the results in the odd prime case follow easily from the existing literature, the $p=2$ case does not appear in the literature. In particular, the exceptional supercuspidal cases are especially difficult since there is no closed-form expression of base change of these representations. In Section 5 we also determine the explicit base change of an exceptional representation of $GL(2,\mathbb{Q}_2)$ with trivial central character and conductor $3,$ in terms of the $\varepsilon$-factors associated to these representations by Langlands. Both sections mentioned above are novel and to the best of our knowledge, have not appeared in the literature before. Using explicit constructions such as that in \cite{MR466025}, bounds on the level of Hilbert modular forms obtained by  Doi--Naganuma lifts have been obtained. However, the exact level of the newform associated via base change has not been addressed before. In \cite{MR3911789}, the authors study the problem of lifting elliptic modular forms of level $N$ to a real quadratic field $K$ of discriminant $D$, under the assumptions $(N,D)=1$ and $D\equiv 1\pmod 4$. Their results agree with ours in this setting; however, our results apply in greater generality. We also remark that similar problems about conductors under functorial transfers of modular forms have been considered before by \cite{MR4356848} and \cite{MR4878181} in the case of $\text{Sym}^3$ transfer. In particular, some of our notation follows \cite{MR4878181}. We also remark that a similar problem could be studied regarding  base change of representations of $GL(n,F_\mathfrak{p}).$ In this case the existence of the conductor of a representation $\pi,$ $c(\pi)$ was dealt with for $n>2$ by \cite{MR620708} when $\pi$ is generic and \cite{MR4522693} for $\pi$ non-generic. Our results in Section 3 completely deal with the case where $K_\mathfrak{P}$ is a tamely ramified extension of $F_\mathfrak{p}.$
\subsection{Main Results}We summarize the main theorems below under certain assumptions. The results in complete generality can be found in Section 6. For a Hilbert modular form $f$ over a totally real field $F,$ we define its \emph{exact level} to be the conductor of the automorphic representation of $GL(2,\mathbb{A}_F)$ associated to $f.$
\begin{theorem}\label{thm:1.1}
Let $F$ be a totally real number field,$f$ a Hilbert modular form over $F$
 of exact level $N$ and nebentypus $\chi.$Let $K$ be a totally real quadratic extension of $F.$For each dyadic prime $\mathfrak{p}$ of $F,$assume that one of the following holds:    \begin{enumerate}
        \item $v_\mathfrak{p}(N)=0,$  or
        \item $\mathfrak{p}$ is unramified in $K.$
    \end{enumerate} Let $\mathcal{R}$ denote all primes dividing $N$ which are ramified in $K,$ and $\mathcal{S}$ denote all primes dividing $N$ which are unramified in $K.$ Additionally assume that when $\mathfrak{p}\in \mathcal{R},$ and $\pi_\mathfrak{p}$ is not supercuspidal, then $v_\mathfrak{p}(N)>2.$  Then the exact level of base change lift of $f$ to $K,$ $f_K$ is given by  $$\mathfrak{N}=\prod_{\substack{\mathfrak{P}\mid \mathfrak{p}\\{\mathfrak{p}\in \mathcal{R}}}} \mathfrak{P}^{2(v_{\mathfrak{p}}(N)-1)}\prod_{\substack{\mathfrak{P}\mid \mathfrak{p}\\ \mathfrak{p} \in \mathcal{S}}}\mathfrak{P}^{v_{\mathfrak{p}}(N)}$$
\end{theorem}
Let us specialize to the case $F=\mathbb{Q}.$
Let $K=\mathbb{Q}(\sqrt{D})$ be of discriminant $D.$

 \begin{theorem}\label{thm:1.2}
Let $F=\mathbb{Q}$ and $K=\mathbb{Q}(\sqrt{D})$ with $D>0$ and $D \equiv 0 \pmod{4}.$ Let $$\mathcal{R}=\{p \in \mathbb{Z}: p\mid D \text{ and } p\mid N\}$$
$$\mathcal{S}=\{p \in \mathbb{Z}: p\nmid D \text{ and } p\mid N\}$$
Let us restrict to the case  $v_p(N)> 2v_p(D)$ for all $p$ in $\mathcal{R}.$ Then
   $$\mathfrak{N}=\prod_{\substack{p\in \mathcal{R}\\ \mathfrak{P}\mid p\\}} \mathfrak{P}^{2(v_p(N)-v_p(D))}\prod_{\substack{p\in \mathcal{S}\\ \mathfrak{P}\mid p}}   \mathfrak{P}^{v_p(N)}.$$
\end{theorem}
As we remarked before, we believe our main contribution is the comprehensive treatment when the local representation is an exceptional supercuspidal representation. We present a simple case of our results in the following. 
\begin{theorem}
   Let $f$ be an elliptic newform over $\mathbb{Q}$ of level $2^n$ such that $\pi_2$ is an exceptional supercuspidal representation. Let $K$ be a quadratic extension of $\mathbb{Q}$ of discriminant $D$ such that $2$ is ramified in $K.$ Let $\mathfrak{P}$ be the prime above $2.$ Assume that $n\neq 2v_2(D).$ Let $f_{K}$ be the newform over $K$, associated to $f$ via base change. Then the exact level of $f_K$ is given by
   $$\mathfrak{N}=\begin{cases}
       \mathfrak{P}^n& \text{if }n<2v_2(D), \\
       \mathfrak{P}^{2(n-{v_2(D)})}& \text{otherwise.}
   \end{cases}$$
\end{theorem} 
Now let us provide an explicit description of local base change in the case where $\pi$ is a supercuspidal representation of $GL(2, \mathbb{Q}_2)$ with $a(\pi) = 3$ and trivial central character. This is the hardest case treated in this paper: when the extension is wildly ramified, explicit base change of an exceptional representation is difficult to pin down, and in general no closed-form description is known. The theorem below resolves it completely for this exceptional representation.

 Let $\pi$ be a representation of $GL(2,F_p)$ with trivial central character, and let  $\varepsilon(\pi)$ be its local root number, which is given by $\pm1.$ Let $K_\mathfrak{P}$ be a ramified quadratic extension of $\mathbb{Q}_2$ and $\pi$ an exceptional supercuspidal representation of $GL(2,K_{\mathfrak{P}}).$ Up to twisting, there are $4$ equivalence classes of $\pi,$  namely one class of each conductor $3$and$5$ and two classes of conductor $7.$ Thus when $a(\pi)=3$ and $\pi$ has trivial central character, then there are two such $\pi$ up to isomorphism which have opposite signs for $\varepsilon(\pi).$ 
\begin{theorem}
Let $\pi$ be an exceptional supercuspidal representation of $GL(2,\mathbb{Q}_2)$ with trivial central character and $a(\pi)=3.$ For a ramified quadratic extension $K_{\mathfrak{P}}$ of $\mathbb{Q}_2,$ $\pi_{K_\mathfrak{P}}$ is the exceptional supercuspidal representation of $GL(2,K_{\mathfrak{P}})$ with $a(\pi_{K_{\mathfrak{P}}})=3$ and $$\varepsilon(\pi_{K_{\mathfrak{P}}})=\begin{cases}
    -\varepsilon(\pi)& \text{ if } d(K_\mathfrak{P}/\mathbb{Q}_2)=2,\\
     \varepsilon(\pi)& \text{ if } d(K_\mathfrak{P}/\mathbb{Q}_2)=3.
\end{cases}$$
\end{theorem}

\section*{{Acknowledgements}}
 I would like to thank my advisor, Kimball Martin, for suggesting this problem and for his invaluable guidance. I am also grateful to him for carefully going through an earlier version of the paper and helping to improve the exposition. 
\section{Some preliminaries}

\subsection{Local representation theory}
Let $F$ be a non-archimedean local field with an additive valuation $v_F,$  $\mathfrak{o}_F$ be the set of elements with non-negative valuation, $\mathfrak{p}$ be the set of elements  with positive valuation and $q_F,$ the residue cardinality. This valuation induces a topology on $F$ with $\mathfrak{p}^{i}$ forming a basis of open neighbourhoods of $0$. We induce a metric on $F,$ by defining $$\vert x\vert_F=q^{-v_F(x)}.$$ 
 Let $W_F$ denote the Weil group of $F$. We have a canonical map, the \emph{Artin homomorphism} from $W_F$ to $F^\times,$ which we denote by $\text{Art},$ given by local class field theory.
In this section, we recall some basic preliminaries of the representation theory of $GL(2,F)$, the group of $2\times 2$ invertible matrices with entries in $F.$ By the Local Langlands Correspondence (abbreviated as LLC), these representations correspond to $2$-dimensional semisimple Weil--Deligne representations of $W_F.$ Thus, we briefly describe both objects.

    A \emph{representation} of $W_F$ is a continuous homomorphism $$\sigma:W_F\rightarrow \text{Aut}_\mathbb{C}(V),$$ where $V$ is a finite-dimensional vector space over $\mathbb{C}.$ A \emph{Weil--Deligne representation} is a pair $(\sigma,N)$ such that $\sigma$ is a representation of $W_F$ and $N$ is a $\mathbb{C-}$endomorphism of $V$ such that for each $x$ in $W_F$ we have $$\sigma(x)N\sigma(x)^{-1}=\vert x\vert N,$$ where $\vert x\vert=\vert \text{Art}^{-1}(x)\vert_F.$

To these representations we associate what are called the $\varepsilon$-factors.  We do not need the full generality of these factors; however, we note some of the important properties of them.  $\varepsilon(\sigma,s,\psi) $is a Laurent polynomial in $q^s$, for a complex number $s.$  This function satisfies the following properties
\begin{itemize}
    \item We have 
    \begin{equation}
    \varepsilon(\sigma,s,\psi)=cq_{F}^{(nc(\psi)-a(\sigma))s},
    \end{equation} 
for some constant $c.$The integer $a(\sigma)$ is uniquely determined, and is called the \emph{Artin conductor} of $\sigma.$
    \item Let $\sigma=\sigma_1\oplus \sigma_2....\sigma_n,$ then $a(\sigma)=a(\sigma_1)+a(\sigma_2)...+a(\sigma_n).$ 
    \item If $\sigma$ is irreducible, then $a(\sigma)\geq \dim \sigma.$
\end{itemize}
The main purpose of introducing these $\varepsilon$-factors is that analogous factors can be defined for representations of $GL(n,F),$ and we require the bijection given by the Local Langlands correspondence to respect these factors, rather than merely being a set-theoretic bijection. We are not going to work with the full representation theory of $GL(n,F)$ since our focus is mostly on modular forms. Thus, we restrict to the case $n=2$ and denote $GL(2,F)$ by $G.$ We now introduce the following subgroups of $G,$ which play an important role in the representation theory of $G.$ 
\begin{align}
 B&=\{\begin{bmatrix}
     a & b \\
     0 & d
 \end{bmatrix}:a,b,d\in F \text{ and } ad\neq0\},\\
 T&=\{\begin{bmatrix}
     a & 0 \\
     0 & d
 \end{bmatrix}:a,d\in F \text{ and } ad\neq0\},\\
\Gamma_{0}(m)&=\{\begin{bmatrix}
     a & b \\
     c & d
 \end{bmatrix}:c\equiv0\pmod{\mathfrak{p}^{m}}\}\subset GL(2,\mathfrak{o}), \text{ and }\\
 \Gamma_{0}(0)&=GL_{2}(\mathfrak{o}).
\end{align}

     For $g=\begin{bmatrix}
     a & b \\
     0 & d
 \end{bmatrix} \in B$, let $$\delta_{B}(g)=(\vert\dfrac{a}{d}\vert)^{\tfrac{1}{2}},$$ the \emph{modular quasicharacter of }$B$.

    Let $\chi$ be a character of $F^{\times}.$ The \emph{conductor} of $\chi,$ denoted by $a(\chi),$ is the smallest integer $i$ such that $\chi$ is trivial on $\mathcal{U}_{i}=1+\mathfrak{p}^{i}.$

 Let $(\pi,V)$ be a complex representation of $G,$ i.e  $V$ is a $\mathbb{C}$-vector space and $\pi:G\rightarrow Aut_{\mathbb{C}}(V)$ is a group homomorphism. The representation $\pi$ is
    \begin{enumerate}
        \item \emph{smooth} if for every $v \in V$, there exists a compact open subgroup $U_{v}$ such that $\pi(u).v=v,$ for all u in $U_{v}.$
        \item \emph{admissible} if $\pi$ is smooth and for any compact open subgroup $K$ of $G,$ the space of $K$-fixed vectors, $\pi^{K}$ is finite dimensional.
    \end{enumerate}

    For an irreducible admissible representation $\pi$ of $G,$ we call the \emph{conductor of $\pi$}, denoted by $a(\pi),$ the smallest integer $m$ such that $\pi^{\Gamma_{0}(m)}\neq0.$ 

We know that every irreducible smooth representation of $GL(n,F)$ is admissible.
LLC gives a canonical bijection between $n$-dimensional Weil--Deligne representations and irreducible admissible representations of $GL(n,F).$ This correspondence preserves the $\varepsilon\textit{-factors},$ and thus the conductors themselves.

From now on we omit the adjectives and simply write representations for irreducible smooth(hence admissible) representations of $GL(2,F).$ We now classify all such representations of $G.$ 
    
\subsubsection{One dimensional representations:} These representations factor through the determinant map. Hence they are of the form $\chi \circ \det,$ where $\chi$ is a character of $F^{\times}.$ These representations do not arise as local components of automorphic forms and hence are irrelevant for our purpose. We also have $a(\chi \circ det)=a(\chi).$ By LLC they correspond to characters of $W_F,$ via local class field theory.
\subsubsection{Parabolically induced representations:} These are representations which can be obtained as a subquotient of a representation induced from a representation of $B.$ They can be parametrized as follows. Let $\chi_{1},\chi_{2}$ be characters of $F^{\times}.$ Then define $\chi:B \rightarrow \mathbb{C^{\times}},$ such that $\chi(\begin{bmatrix}
    a&b\\
    0&d
\end{bmatrix})=\chi_{1}(a)\chi_{2}(d).$ Let us denote 
$$B(\chi_{1},\chi_{2})=\{f:G\rightarrow\mathbb{C};f(b.g)=\chi(b)\delta_{B}(b)f(g), \forall b \in B, \forall g \in G\}$$ on which $G$ acts via right translations. The representation $B(\chi_{1},\chi_{2})$ is \begin{enumerate}
    \item irreducible if $\chi_1\chi_2^{-1}\neq \vert.\vert^{\pm}$. This is called a \emph{principal series representation}. We denote this representation by $\pi(\chi_1,\chi_2).$
    \item admits a unique irreducible subrepresentation/subquotient otherwise. Such a representation is called \emph{special representation}. We denote the unique irreducible subquotient of $B(\vert.\vert^{\tfrac{1}{2}},\vert.\vert^{-\tfrac{1}{2}})$ by ${St}.$
\end{enumerate} We also have $a(\pi(\chi_{1},\chi_{2}))=a(\chi_{1})+a(\chi_{2})$ and $$a({St}\otimes \chi)=\begin{cases}
    1& \chi \text{ is unramified}\\
    2a(\chi)& \chi \text{ is ramified}\\
\end{cases}$$ By LLC they correspond to the reducible $2$-dimensional representations of $W_F.$
\subsubsection{Supercuspidal representations} These are representations of $G$ which are not of the above two types. We refer the reader to \cite{MR2234120} for an explicit construction of these representations. These correspond to the irreducible $2$-dimensional representations of $W_F.$ They can be divided into two types.
\paragraph{\emph{Dihedral supercuspidal representations:}} These representations can be parametrized as $\pi(L,\varkappa),$ where $L$ is a quadratic extension of $F,$ and $\varkappa$ is a character of $L^{\times}$ such that $\varkappa\neq \varkappa \circ \sigma$ where $\sigma$ is the non-trivial $F$-automorphism of $L.$ By \cite{MR476703} Lemma 4.4 this also means that $a(\varkappa)\geq d(L/F).$ By LLC these correspond to the representation $\text{Ind}^{W_F}_{W_L}(\varkappa).$ For such a representation $a(\pi(L,\varkappa))=f(L/F)a(\varkappa)+d(L/F),$ where $f(L/F)$ is the degree of the residue class field extension of $L/F$ and $d(L/F)$ is the valuation of the discriminant of the extension. This means that $$a(\pi)=\begin{cases}
    2a(\varkappa)& L \text{ is unramified,}\\
    a(\varkappa)+d(L/F)& L\text{ is ramified.}\\
\end{cases}$$
\paragraph{\emph{Exceptional supercuspidal representations:}}These representations occur only when $F$ is dyadic. The construction of these representations can be found in \cite{MR734781}. We defer their precise  definition to \cref{def:exceptional}. By LLC they correspond to $2$-dimensional representations of the Weil group, which are not induced. The explicit description of this correspondence for the case $F=\mathbb{Q}_2$ can be found in \cite{MR506171} Chapter 3.
\subsection{Base change of representations}
There are various ways to define base change.  As indicated in the introduction there are two components, local base change and the global base change.
\begin{samepage}
    \paragraph{Local base change}
Let $F$ be a local field, $K$ be an extension of $F$. Let $\lambda_F$ (resp. $\lambda_K$) be the respective maps from irreducible admissible representations of $GL_n(F)$ (resp. $GL_n(K)$) to $n$-dimensional representations of $W_F$ and $W_K$, respectively. Let $\pi$ be an irreducible admissible representation of $GL_n(F)$. Then $\text{Ift}_K(\pi)$ (denoted by $\pi_K$) is an irreducible admissible representation of $GL_n(K)$ such that
\[
\lambda_K\bigl(\text{Ift}_K(\pi)\bigr)
=
\lambda_F(\pi)\vert_{W_K}.
\]
\[
\begin{tikzpicture}[baseline=(current bounding box.center)]
\node (irrF) at (0,1) {$\mathcal{A}_n(F)$};
\node (wdF) at (3,1) {$\mathcal{S}_n(F)$};
\node (irrK) at (0,0) {$\mathcal{A}_n(K)$};
\node (wdK) at (3,0) {$\mathcal{S}_n(K)$};

\draw[->] (irrF) -- node[above] {$\lambda_F$} (wdF);
\draw[->] (irrF) -- node[left] {$\text{lft}_K$} (irrK);
\draw[->] (irrK) -- node[below] {$\lambda_K$} (wdK);
\draw[->] (wdF) -- node[right]{$\vert_{W_K}$}(wdK);
\end{tikzpicture}\]
\end{samepage}
\paragraph{Global Base Change.}
    Let $\pi=\bigotimes_{v}\pi_{v}$ be an automorphic representation of $GL(n,\mathbb{A}_F).$ Then $\Pi=\bigotimes_{w}\Pi_w$ is a $K/F$-lift of $\pi,$ if for all $w\mid v, \Pi_w=\text{lft}_{K_w}(\pi_v).$

\begin{remark}
The fact that the definition above agrees with our earlier definition when $n=2,F=\mathbb{Q}, K$ is a quadratic extension follows from strong multiplicity one for $GL(n).$ 
\end{remark}

Here let us summarize the main results for base change of representations of $G$ to a quadratic extension $K/F.$ Our approach follows \cite{MR546613} where the author also gives a sketch of the proof. A more complete treatment can be found in \cite{MR574808} and in \cite{MR1007299} where the authors consider general $GL(n).$ In what is to follow let us fix a quadratic extension $K$ and let $\pi_{K}$ denote the base change of $\pi$ to $GL(2,K).$
For a character $\chi$ of $F^{\times}$ we write $\chi_{K/F}$ for the character $\chi\circ N_{K/F}.$ We can regard the correspondence $\chi \rightarrow \chi_{K/F},$ as a base change of 1-dimensional representations of the Weil group $W_{F}.$ 
\begin{theorem}
    \begin{enumerate}
        \item For $\pi=\pi(\chi_{1},\chi_{2})$ we have $\pi_{K}=\pi((\chi_{1})_{K/F},(\chi_{2})_{K/F}).$
        \item For $\pi=\text{St}\otimes \chi$ we have $\pi_{K}=\text{St}_K\otimes\chi_{K/F}.$
        \item For $\pi=\pi(L,\varkappa)$ we have $$\pi_{K}=\begin{cases}
            \pi(\varkappa,\varkappa\circ \sigma)& \text{ if } L\cong K,\\
            \pi(LK,\varkappa_{LK/L}) & \text{ if } L\ncong K.\\
        \end{cases}$$
        \item For $(\pi \otimes \chi)_K=\pi_K\otimes \chi_K.$
        \item $\pi_K\cong \pi^\prime_K \iff$ there exists a character $\chi$ of $F^\times$ trivial on $N_{K/F}(K^\times)$ such that $\pi \cong \pi^\prime \otimes \chi.$
    \end{enumerate}
\end{theorem}
\section{Tamely ramified case}
Throughout this section, let $K/F$ be a tamely ramified extension of local fields of degree $m,$ ramification index $e$ and residue field degree $f.$ For an additive character $\psi$ of the field $F,$ we define the conductor of $\psi,$$n(\psi),$ to be the smallest integer such that $\psi$ is trivial on $\mathfrak{p}^{n(\psi)}.$We now take $\psi$ to have conductor $0.$ We begin with a basic lemma on tamely ramified extensions.
\begin{lemma}
 
       $n(\psi \circ Tr_{K/F})=(e-1).$

\end{lemma}
\begin{proof}
    This is just \cite{MR427267} page 142.
\end{proof}
\begin{theorem}\label[theorem]{thm:tamewd}
    Let $\sigma$ be an irreducible Weil--Deligne representation of dimension $n.$ Assume that $\sigma_K$ is not unramified.(This assumption fails when $n=1$ and $\sigma$ is the character associated to $K/F$ or its unramified twist, by local class field theory.) Then $a(\sigma_K)=ea(\sigma)-(e-1)n.$
\end{theorem}
\begin{proof}
    This follows from Theorem 1.7 in \cite{MR1685898}, with $\pi_2$ taken to be the trivial  character, when $n$ is a power of the residue characteristic.   The integer $f(\pi,\psi)$ in their theorem is $nc(\psi)-a(\sigma)$ written in (1). Otherwise, the result follows from \cite{MR1007299},{Proposition 6.9}.
    \begin{eqnarray}
     nc(\psi_K)-a(\sigma_K)=e(n(\psi)-a(\sigma))\\
     n(e-1)-a(\sigma_K)=-ea(\sigma)\\
     a(\sigma_K)=ea(\sigma)-n(e-1)
    \end{eqnarray}
    
   \end{proof}
\begin{corollary}\label[corollary]{cor:unbase}
       Let $\sigma$ be a representation of $W_F$ of dimension $n.$ Let $K$ be an unramified extension of $F.$ Then $$a(\sigma_K)=a(\sigma).$$ 
\end{corollary}
   We now turn to the representations of $W_F$ corresponding to the irreducible admissible representations of $GL(2)$. More precisely we need to look at the case $\dim \sigma=2$ and $K/F$ is a ramified quadratic extension.
\begin{theorem}\label[theorem]{thm:tame}
Let $K/F$ be a tamely ramified quadratic extension, and let $\pi$ be an admissible representation of $G.$
\begin{enumerate}
    \item If $\dim \pi=1, $ and $\pi_K$ is not unramified, then
   $$a(\pi_K)=
   2a(\pi)-1.$$
    \item

    If $\pi$ is a principal series representation, we assume $w_\pi=1.$ Then $$a(\pi_K)=\begin{cases}
      0&a(\pi)=0 \text{ or } \pi_K \text{ is unramified,}\\
      2a(\pi)-2&\text{ otherwise.}
    \end{cases}$$
   \item If $\pi$ is a special representation, then $$a(\pi_K)=\begin{cases}
       1&\text{ if } a(\pi)=1,\\
       1&\text{ if } a(\pi)=2 \text{ and } \pi\text{ has trivial central character } ,\\
       
       2a(\pi)-2 &\text{ otherwise.}\\       
   \end{cases}$$ 
   \item If $\pi$ is supercuspidal then $$a(\pi_K)=2a(\pi)-2.$$
\end{enumerate}
\end{theorem}
The results of this section are summarized in \cref{tab:tame}. 
\begin{proof}
    For the first and last case use \ref{thm:tamewd} for Weil--Deligne representations of dimension $1$ and $2.$ For the other two cases this follows from a case-by-case calculation using $(1)$ above. When $\pi \cong \text{St}\otimes\chi,$   $a(\pi)=2$ implies $ a(\chi)=1.$ Since $d(K/F)=1$, we have $U^1_F\subset N_{K/F}(F^\times).$ Thus $$a(\chi \circ N_{K/F})=\begin{cases}
        0& N_{K/F}(\mathfrak{o}^\times_K)\subset U^1_F,\\
        1 &\text{otherwise.}
    \end{cases}$$
    The first case happens only when $\chi=w_{K/F}$ on $\mathfrak{o}^\times_F.$ Thus $$a(\pi_K)=\begin{cases}
        1&\chi=w_{K/F}\text{ on } \mathfrak{o}^\times_F,\\
        2&\text{otherwise.} 
    \end{cases} $$
 The case $a(\pi)=2$ and $\chi=w_{K/F}\text{ on } \mathfrak{o}^\times_F$ means that $\pi$ has trivial central character.
\end{proof}

\begin{table}
    \centering
    \begin{tabular}{|c|c|c|}
    \hline
        $\pi$ &Notes  &$a(\pi_K)$ \\\hline
        Principal series representation&$a(\pi)= 0$     & $a(\pi)$ \\ \hline
         Principal series representation&$a(\pi)= 2, \pi_K \text{ is not unramified}$  & $a(\pi)$ \\ \hline
         $\pi(\chi_1,\chi_2)$&$a(\pi)> 2$  & $2(a(\pi)-1)$ \\ \hline
         Special representation&$a(\pi)=1 \text{ or } \pi \text{ has trivial central character}$&1\\ \hline
         Special representation&$\pi \text{ is not of the form above}$&$2(a(\pi)-1)$\\ \hline
         \text{ supercuspidal}&  &$2(a(\pi)-1)$ \\ \hline
    \end{tabular}
    \caption{Summary of results when $F$ is a $p$-adic field for odd $p$, $K/F$ is a ramified quadratic extension (When $\pi$ is principal series, we assume $\pi$ has trivial central character)}
    \label{tab:tame}
\end{table}
\section{Imprimitive representations of \texorpdfstring{$\mathbb{Q}_{2}$}{Q2}}
Let us now shift our focus to the case where $F$ is dyadic. As remarked in the introduction, this case is more involved than the previous. Let $F$ be a dyadic field, $K$ a quadratic extension of $F.$ Let $\varpi_{K}$ be a uniformizer of $K$ and let $v_K$ be a valuation on $K$ such that $v_K(\varpi_K)=1.$  Let $G=\operatorname{Gal}(K/F)$ be the Galois group of this extension, generated by $\tau.$ It is easy to see that the discriminant $d(K/F)=v_{K}(\tau(\varpi_{K})-\varpi_{K}),$ which we write as $d$ if it is clear from the context which extension we are dealing with. It can also be seen by class field theory that $d(K/F)$ is the largest integer such that $U^{d(K/F)}_F$ is contained in the image of the norm map, $N_{K/F}.$ Also we can see that $$d(K/F)=\begin{cases}
    0&K/F \text{ is unramified,}\\
    1&K/F \text{ is ramified and $p$ is odd,}\\
    >1&K/F \text{ is ramified and $p=$}2.\\
\end{cases}$$ 

\begin{proposition}\label[proposition]{prop:norm}
\cite{MR554237},  Page 85 Corollary 3
    
    \begin{enumerate}
    \item $N_{K/F}(\mathfrak{o}_{K}^{\times})\subset \mathfrak{o}^{\times}_{F}$
        \item  $N_{K/F}(U^{\psi(n)}_{K})=U^{n}_{F}$ for $n\geq d(K/F)$.
   \item $N_{K/F}(U^{d+2i}_{K})=N_{K/F}(U^{d+1+2i}_{K})= U^{d+i}_{F}, \forall i\geq0$
\end{enumerate}
where $\psi(n)=\begin{cases}
    n& n\leq d-1\\
    2n-(d-1) &n\geq d-1\\
\end{cases}$
\end{proposition}
Let us also write  a preliminary lemma on how to deduce the discriminant and the degree of residue field extensions on towers.
    \begin{lemma}\label[lemma]{lem:disc}
        Let $F$ be a local field and let $L_{1}, L_{2}$ be two quadratic extensions of $F.$ Then $L_{1}L_{2}$ is a $\mathbb{Z}/2\mathbb{Z}\times \mathbb{Z}/2\mathbb{Z}$ extension of $F.$ Then we have the following lattice of extensions.   
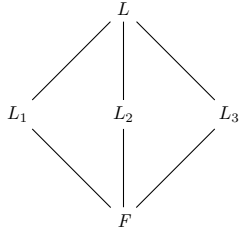
\[
\begin{tikzpicture}[node distance = 2cm, auto,scale=0.7, transform shape]
\node (a) {F};
\node (b) [above of=a, left of=a] {$L_{1}$};
\node (c) [above of=a] {$L_{2}$};
\node (d) [above of=a, right of=a] {$L_{3}$};
\node (e) [above of=a, node distance = 4cm] {$L$};
\draw[-] (a) to (b);
\draw[-] (a) to (d);
\draw[-] (b) to (e);
\draw[-] (d) to (e);
\draw[-] (a) to (c);
\draw[-] (c) to (e);
\end{tikzpicture}
\] 
Then for $i=1,2,3$ \begin{enumerate}
    \item $f(L/L_{i})f(L_{i}/F)=f(L/F),$
    \item $d(L/F)=\sum_{i=1}^{3} d(L_{i}/F)$ and
    \item $d(L/F)=2d(L_{i}/F)+f(L_{i}/F)d(L/L_{i}).$
    \item $d(L_3/F)\leq \max \{d(L_1/F),d(L_2/F)\},$ with equality if $d(L_1/F)\neq d(L_2/F).$ 
    \item In $(4)$ above, the inequality is strict if $F$ is a totally  ramified extension of $\mathbb{Q}_2.$
\end{enumerate}
    \end{lemma}
    \begin{proof}
        $(1)$ and $(3)$ just follow from the behavior of residue class degree and discriminant in towers. $(2)$ follows from the conductor-discriminant formula or \emph{F\"uhrerdiskriminantenproduktformel} by Hasse. $(4)$ is clear when one of the extensions is unramified. Now let $\chi_i$ be the characters of $F^\times$ corresponding to $L_i$. Clearly $\chi_3=\chi_1\chi_2$. The case where $d(L_1/F)\neq d(L_2/F),$ implies that $\chi_i$ is trivial on $U^{d(L_i/F)}_F$ for $i=1,2.$ This means that $\chi_3$ is trivial on the largest unit group among them, and $(4)$ follows. Now for $(5)$ if $d_1=d_2=d,$ we have both $\chi_1$ and $\chi_2$ are trivial on $U^d_F,$ and both of them are trivial on $U^{d-1}_F.$ Since $F$ is totally ramified over $\mathbb{Q}_2,$ we get $U^{d}_F$ is an index two subgroup in $U^{d-1}_F.$ Thus we get $\chi_1(u)=\chi_2(u)=-1,$ when $u$ is in $U^{d-1}_F$ but not $U^{d}_F.$ Thus $\chi_1\chi_2$ is trivial on $U^{d-1}_{F},$ hence $d(L_3/F)<\max \{d(L_1/F),d(L_2/F)\}.$
        
        All of these statements can be found in \cite{MR1697859} although $d(L/F)$ is the $p-$adic valuation of the discriminant in  \cite{MR1697859}.
    \end{proof}
While the above lemma is obvious for odd primes since there is only one $\mathbb{Z}_{2}\times \mathbb{Z}_{2}$ extension up to isomorphism we record it for use when we have to deal with dyadic extensions. From now on let us assume that $F$ is a totally ramified extension of $\mathbb{Q}_2.$ Applying Proposition 1.5, page 72 in \cite{MR1915966}to $F,$ we get the following proposition.
\begin{proposition}

\medskip

\begin{tikzpicture}[baseline=(current bounding box.center)]
\node (OK) at (0,1) {$\mathfrak{o}_K^{\times}$};
\node (KF) at (3,1) {$\overline{K}^{\times}$};
\node (OF) at (0,0) {$\mathfrak{o}_F^{\times}$};
\node (FF) at (3,0) {$\overline{F}^{\times}$};

\draw[->] (OK) -- node[above] {$\lambda_0$} (KF);
\draw[->] (OK) -- node[left] {$N_{K/F}$} (OF);
\draw[->] (KF) -- node[right] {$\mathrm{id}$} (FF);
\draw[->] (OF) -- node[below] {$\lambda_0$} (FF);
\end{tikzpicture}

\medskip

\begin{tikzpicture}[baseline=(current bounding box.center)]
\node (UK) at (0,1) {$U_K^{i}$};
\node (KF) at (3,1) {$\overline{K}$};
\node (UF) at (0,0) {$U_F^{i}$};
\node (FF) at (3,0) {$\overline{F}$};

\draw[->] (UK) -- node[above] {$\lambda_i$} (KF);
\draw[->] (UK) -- node[left] {$N_{K/F}$} (UF);
\draw[->] (KF) -- node[right] {$\mathrm{id}$} (FF);
\draw[->] (UF) -- node[below] {$\lambda_i$} (FF);
\end{tikzpicture}
\quad \textit{for } $1 \le i < d-1$.

\medskip

\begin{tikzpicture}[baseline=(current bounding box.center)]
\node (UK) at (0,1) {$U_K^{d-1}$};
\node (KF) at (3,1) {$\overline{K}$};
\node (UF) at (0,0) {$U_F^{d-1}$};
\node (FF) at (3,0) {$\overline{F}$};

\draw[->] (UK) -- node[above] {$\lambda_{d-1}$} (KF);
\draw[->] (UK) -- node[left] {$N_{K/F}$} (UF);
\draw[->] (KF) -- node[right] {$0$} (FF);
\draw[->] (UF) -- node[below] {$\lambda_{d-1}$} (FF);
\end{tikzpicture}

\end{proposition}

This means we have
\begin{enumerate}
    \item $N_{K/F}(\mathfrak{o}^{\times}_{K}) \subset \mathfrak{o}^{\times}_{F}.$
    \item $N_{K/F}(U^{i}_{K}) \subset U^{i}_{F}$ and $N_{K/F}(U^{i}_{K}) \not\subset U^{i+1}_{F}$ for $1\leq i<d-1.$
    \item $N_{K/F}(U^{d-1}_{K}) \subset U^{d}_{F}.$
\end{enumerate}
\begin{corollary}\label[corollary]{Characters}
    Let $K/F$ be a ramified quadratic extension of local fields, and let $\varkappa$ be a character of $F^{\times}.$ Then
    $$a({\varkappa_K})=\begin{cases}
        a(\varkappa) & 0\leq a(\varkappa)< d,\\
        <d & a(\varkappa)=d,\\
        2a(\varkappa)-d          & \text{otherwise}.
    \end{cases}$$
\end{corollary}
\begin{proof}
    The case $a(\varkappa)>d$ is clear. Now for $n=a(\varkappa)< d$ the proposition implies that $a(\varkappa_{K})\leq n.$ Now if $a(\varkappa_{K})<n$ this means $\varkappa$ is trivial on both $N_{K/F}(U^{n-1}_{K})$ and $U^{n}_F.$ However these two subgroups are distinct index two subgroups of $U^{n-1}_{F}$ unless $N_{K/F}(U_K^{n-1})=U_F^{n},$ which happens when $n=d(K/F).$ But we have $n<d(K/F),$ hence $\varkappa$ is trivial on $U^{n-1}_{F}$ contradicting $a(\varkappa)=n.$ Hence $a(\varkappa_{K})=n,$ when $a(\varkappa)<d.$ See \cref{Prop:char} for the case $a(\varkappa)=d$ when $F=\mathbb{Q}_2.$
\end{proof}
\begin{remark}
    The case of $a(\varkappa)>d$ is also true for general dyadic fields by Proposition 4.1 above.
\end{remark}

Now let us make a brief note on the arithmetic of quadratic fields over dyadic fields. Let $e$ be the ramification index of $K/F.$ Then the following are the quadratic extensions of $F,$ with the number of them. See Lemma 4.3 in \cite{MR476703}.
\begin{itemize}
    \item A unique unramified extension.
    \item $2^{k}$ quadratic extensions each of discriminant $2k,$ where $1 \leq k\leq e.$
    \item $2^{e+1}$ quadratic extensions of discriminant $2e+1.$
\end{itemize}For $F=\mathbb{Q}_{2}$ this means that there are $7$ quadratic extensions of $F$ as follows:\begin{enumerate}
    \item $\mathbb{Q}_{2}(\sqrt{5}),$ the unramified extension.
    \item $\mathbb{Q}_{2}(\sqrt{3})$ and $\mathbb{Q}_{2}(\sqrt{-1}),$ with discriminant 2.
    \item $\mathbb{Q}_{2}(\sqrt{2})$ and $\mathbb{Q}_{2}(\sqrt{10}),$ with discriminant 3.
    \item $\mathbb{Q}_{2}(\sqrt{-2})$ and $\mathbb{Q}_{2}(\sqrt{-10}),$ with discriminant 3.
\end{enumerate}
The pair of fields in each entry differ by the unramified character of $F^{\times}.$ For any other pairings of quadratic fields with the discriminant $3,$ they differ by a character of conductor $2.$ 
From now on let us set $F=\mathbb{Q}_2.$
Let us explain the case when $a(\chi)=d(K/F), \text{ for } F=\mathbb{Q}_2.$ Thus $a(\chi)=2 \text{ or } 3.$
\begin{proposition}\label[proposition]{Prop:char}
Let $K$ be a ramified quadratic extension of $F$ with discriminant $d(K/F)$ and let $\chi$ be a character of $F^\times$ with $a(\chi)=d(K/F).$Then, 
    $$a(\chi_K)=\begin{cases}
 0       & \chi=w_{K/F} \text{ on } \mathfrak{o}_F^\times,\\ 
 a(\chi)-1 & \text{otherwise.}   
 \end{cases}$$
\end{proposition}
\begin{proof}
    Let $a(\chi)\leq n,$ then $\chi$ is a character of $\mathfrak{o}^{\times}_{F}/U_{n}$. The latter is a finite group of cardinality $2^{n-1}.$ Therefore the number of characters of conductor$=n$ is given by $2^{n-1}-2^{n-2}=2^{n-2},$ upto unramified twist. When $n=2$ there is only a single character upto unramified twist, which we call $\chi_2$. Since there are two quadratic extensions of $F$ with discriminant $2,$ we get $a(\chi_K)=0.$ When $a(\chi)=3$ we see that there are $2$ characters upto unramified twist. Clearly $w_{K/F}$ is one and $\chi=\chi_2.w_{K/F}$ is the other. The case when $\chi=w_{K/F}$ is clear, and for the second one we get $\chi\circ N_{K/F}=\chi_{2}\circ N_{K/F}$ on $\mathfrak{o}_F^\times.$ Hence we get $a(\chi_K)=2=a(\chi)-1.$
\end{proof}
\begin{remark}
    The second case in the previous proposition only happens when $d(K/F)=3.$
\end{remark}
\begin{theorem}
    Let $\pi$ be a representation of $GL(2,\mathbb{Q}_2)$ which is not an exceptional representation. Let $K$ be a ramified quadratic extension of $\mathbb{Q}_2.$ We further assume that when $\pi$ is  principal series representation,$\pi$ has trivial central character. Then  
    \begin{enumerate}
        \item Let $a(\pi)<2d(K/F).$ Furthermore if $\pi$ is supercuspidal assume that $\pi$ is not induced from the unramified quadratic extension of $F.$ Then
        $$a(\pi_K)=a(\pi).$$
        
        \item When  $a(\pi)>2d(K/F),$ 
          $$a(\pi_K)=2(a(\pi)-d(K/F)).$$
        \item When $a(\pi)=2d(K/F)$ and $\pi$ is not supercuspidal. Then
        $$a(\pi_K)=\begin{cases}
            0&\pi \text{ is principal series and }\chi=w_{K/F} \text{ on } \mathfrak{o}^\times_F.  \\
            1&\pi \text{ is a special representation and } \chi=w_{K/F} \text{ on } \mathfrak{o}^\times_F. \\
            a(\pi)-2&\text{otherwise.}
        \end{cases}$$
        The last case only happens when $d(K/F)=3$ by the previous remark.
        \item When $a(\pi)=2d(K/F)$ and $\pi$ is supercuspidal assume that 
        \begin{enumerate}
        \item $L$ is not unramified, or
        \item When $d(K/F)=3,d(L/F)\neq 2.$
        \end{enumerate}
        Then $$a(\pi_K)=a(\pi)=2(a(\pi)-d(K/F)).$$
        \end{enumerate}
\end{theorem}
\begin{proof}
We do this case by case.
  
        When $\pi$ is principal series we have  $\pi =  \pi(\chi,\chi^{-1})$ and $\pi_K= \pi(\chi_K,\chi^{-1}_K).$ Thus $a(\pi_K)=2a(\chi_K)$ and the theorem follows from \cref{Characters} and for the case $a(\chi)=d$ from \cref{Prop:char} that $$a(\pi_K)=2{a(\chi_ K)}=\begin{cases}
            2a(\chi)& a(\chi)<d(K/F)\\
             0       &a(\chi)=d(K/F) \text{ and } \chi=w_{K/F} \text{ on } \mathfrak{o}_F^\times,\\ 
            2(a(\chi)-1)       &a(\chi)=d(K/F) \text{ and } \chi\neq w_{K/F} \text{ on } \mathfrak{o}_F^\times,\\ 
            2(2a(\chi)-d)& a(\chi)>d(K/F).
        \end{cases}.$$ 
        Noting $a(\pi)=2a(\chi)$ gives the result.
        Now when $\pi$ is a special representation we have $\pi= \text{St}\otimes \chi.$ This is similar to the previous case noting that $a(\pi)=\max\{1,2a(\chi)\}.$

        Now let us deal with the case $\pi= \pi(L,\varkappa)$ where $L$ is the unramified quadratic extension. We assume that $a(\varkappa)>d(K/F),$ thus $a(\pi)>2d(K/F).$ We have $\pi_K= \pi(LK,\varkappa_{LK/L}).$ Since $L$ is unramified and $K$ is ramified, we have $LK/K$ unramified and hence $a(\pi_{K})=2a(\varkappa_{LK/L}).$ Using \cref{lem:disc}, we get $d(LK/L)=d(K/F).$ The result just follows from noting that $a(\pi)=2a(\varkappa)$ and $a(\varkappa)>d(K/F).$ 

      Let us look at the case where $\pi= \pi(L,\varkappa)$ where $L\cong K.$ We clearly have $a(\pi)\geq 2d(K/F).$ Since $\pi_K=\pi(\varkappa,\varkappa\circ \sigma)$ and thus $a(\pi_K)=a(\varkappa)+a(\varkappa\circ \sigma)=2a(\varkappa)=2(a(\pi)-d(K/F)).$

         Now suppose $\pi=\pi(L,\varkappa)$ where $L$  differs from $K$ by the unramified quadratic character. Since $d(L/F)=d(K/F),a(\pi)\geq2d(K/F).$Then $\pi_{K}=\pi(LK,\varkappa_{LK}).$ Since both $L$ and $K$ differ by the unramified character, we have $LK/K,$ and $LK/L$ unramified, and hence $a(\pi_{K})=2a(\varkappa_{LK})=2a(\varkappa)=2(a(\pi)-d(L/F)).$ The result just follows from noting that $ d(K/F)=d(L/F).$

         Finally, when $\pi=\pi(L,\varkappa)$ and $L$ and $K$ are both ramified and $\pi$ is not one of the cases above. Let us start with $d(K/F)=3.$ 
     Using \cref{Prop:char}, it can be easily seen that both $d(LK/K)=2$ and $d(LK/L)=8-2d(L/F),$ thus both are ramified. From Corollary 4.4, we have $$a(\varkappa_{LK})=\begin{cases}
    a(\varkappa)& a(\varkappa)<d(LK/L),\\
    2a(\varkappa)-d(LK/L)& a(\varkappa)>d(LK/L).\\
\end{cases}$$
Now $a(\pi_{K})=a(\varkappa_{LK})+2$ and $a(\varkappa)=a(\pi)-d(L/F).$ Putting all of this together we have the result.
When $d(K/F)=2,$ the condition on $L$ implies $d(L/F)=3$ and, similar to previous case we get $$a(\pi_K)=2(a(\pi)-d(K/F)).$$ 
\end{proof}

The results so far are summarized below in  \cref{tab:dyadic}.
\begin{table}
    \centering
    \small
    \setlength{\tabcolsep}{6pt}
    \renewcommand{\arraystretch}{1.15}

    \begin{tabularx}{0.93\textwidth}
            {|l|>{\raggedright\arraybackslash}X|c|}
    \hline
        $\pi$ & Conditions & $a(\pi_K)$ \\
    \hline
        Non supercuspidal representation
        & $a(\pi)<2d(K/F)$
        & $a(\pi)$ \\
    \hline

        Supercuspidal
        & $\pi=\pi(L,\varkappa)$, $a(\pi)<2d(K/F)$, and $L$ is not unramified
        & $a(\pi)$ \\
    \hline
        Principal series
        & $a(\pi)=2d(K/F)$ and $\pi_K$ is not unramified
        & $a(\pi)-2$ \\
    \hline
        $St\otimes\chi$
        & $a(\pi)=2d(K/F)$ and $\chi=w_{K/F}$ on $\mathfrak{o}_F^\times$
        & $1$ \\
    \hline
        $St\otimes\chi$
        & $a(\pi)=2d(K/F)$ and $\chi\neq w_{K/F}$ on $\mathfrak{o}_F^\times$
        & $a(\pi)-2$ \\
    \hline
        Any representation
        & $a(\pi)>2d(K/F)$
        & $2(a(\pi)-d(K/F))$ \\
    \hline
        Supercuspidal with $\pi=\pi(L,\varkappa)$
        & $a(\pi)=2d(K/F)$ and $L$ is not unramified,
          or $d(K/F)=3$ and $d(L/F)\ne2$
        & $2(a(\pi)-d(K/F))$ \\\hline
    \end{tabularx}
    \caption{Summary of results when $F$ is a totally ramified dyadic field  $K/F$ is a ramified quadratic extension (When $\pi$ is principal series, we assume $\pi$ has trivial central character)}
    \label{tab:dyadic}
\end{table}    

    \begin{remark}\label[remark]{rmk:conductor}
        The only remaining cases are 
        \begin{enumerate}
        \item  $a(\pi)=6,$ $\pi=\pi(L,\varkappa)$  $d(L/F)=2,d(K/F)=3$ 
        This corresponds to $d(LK/L)=4=a(\varkappa).$ Similar to the previous propositions about characters of conductor $2$ we will describe how to deal with this case. Let $\chi_2,\chi_3,\chi_4$ be characters of conductor $2,3$ and $4$ respectively. Clearly $\chi_2\chi_4,\chi_3\chi_4$ and $\chi_2\chi_3\chi_4$ are all characters of conductor 4. Since \cref{Prop:char} implies we have $2^{n-2}$ characters of conductor $n,$ these are all characters of conductor $4.$ Thus $a(\chi_K)=0,2,3$ respectively and hence $a(\pi_K)=a(\chi_K)+d(LK/L)=4,6 \text{ or } 7$ respectively.
        \item $a(\pi)<2d(K/F),\pi=\pi(L,\varkappa),$ $L$ is unramified. We have to exclude this case since \cref{Characters} only applies to totally ramified extensions of $\mathbb{Q}_2.$
        \end{enumerate}
\end{remark}

\section{Primitive representations of \texorpdfstring{$\mathbb{Q}_{2}$}{Q2}}
Let us begin with a couple of facts about supercuspidal representations of $G,$ all of which can be found in \cite{MR2234120} Chapter 10. 
\begin{proposition}

    Let $\pi$ be a supercuspidal representation of $G.$ Write $\mathcal{I}(\pi)=\{\chi:F^\times \rightarrow \mathbb{C}^\times; \chi \bigotimes\pi\cong\pi\}.$ Then we have \begin{enumerate}
        \item $\chi \in \mathcal{I}(\pi)$ satisfies $\chi^2=1.$
        \item $\mathcal{I}(\pi)$ is a group of size dividing $4.$
        \item $\omega_{E/F}\in\mathcal{I}(\pi) \iff \pi \cong \sigma(\varkappa,E), $ for some $\varkappa:E^\times\rightarrow\mathbb{C}^\times.$
    \end{enumerate}
\end{proposition}
\begin{definition}\label[definition]{def:exceptional}
    Let $\pi$ be as above, then we say $\pi$ is 
    \begin{enumerate}
        \item \emph{exceptional} if $|\mathcal{I}(\pi)|=1.$
        \item \emph{simply imprimitive} if  $|\mathcal{I}(\pi)|=2.$
         \item \emph{triply imprimitive} if  $|\mathcal{I}(\pi)|=4.$
    \end{enumerate}
    Exceptional representations are further of two types
    \begin{itemize}
    \item Tetrahedral type: There is a Galois cubic extension $L/F$ such that $\pi_L$ is triply imprimitive. 
    \item Octahedral type: There is a non-Galois cubic extension $L/F$ such that $\pi_L$ is imprimitive. Moreover if $\tilde{L}$ is the Galois closure of $L$ then $\pi_{\tilde{L}}$ is triply imprimitive.
    \end{itemize}
\end{definition}
We restrict to the case $F=\mathbb{Q}_2$. 
\begin{proposition}
If $\pi$ is a minimal exceptional representation of $G,$ then $a(\pi)=3,5 \textit{ or }7.$ If $\pi$ has trivial central character then $a(\pi)=3 \textit{ or } 7$ and these representations are octahedral. Moreover when $a(\pi)=5,$ $\pi$ is tetrahedral.
\end{proposition}
\begin{proof}
    See \cite{MR4705658} Proposition 3.9 and \cite{MR476703} Theorem 5.4.
\end{proof}

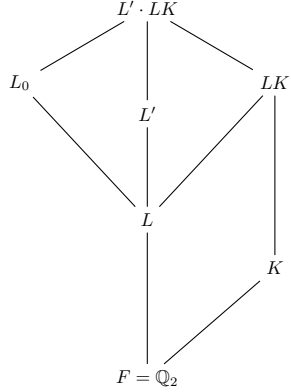
\begin{center}

\begin{tikzpicture}[node distance=2cm,scale=0.7, transform shape]

\node (top) {$L' \cdot LK$};

\node (lprime) [below of=top, yshift=-0.001cm ] {$L^\prime$};
\node (l0) [below left of=top, xshift=-1cm] {$L_0$};

\node (lk1) [below right of=top,xshift=1cm] {$LK$};

\node (l) [below of=lprime] {$L$};

\node (k) [below of=lk1, yshift=-1.5cm] {$K$};

\node (fq) [below of=l, yshift=-1cm] {$F=\mathbb{Q}_2$};

\draw (top) -- (lprime);
\draw (top) -- (l0);
\draw (l0) -- (l);
\draw (top) -- (lk1);
\draw (lprime) -- (l);
\draw (l) -- (fq);
\draw (lk1) -- (k);
\draw (l) -- (lk1);
\draw (fq) -- (k);
\end{tikzpicture}
\end{center}

Here $L$ is a cubic extension of $F$ and $K$ is a ramified quadratic extension of $F,$ and we aim to completely determine $\pi_K.$ It is easy to see that $\pi_K$ again is an exceptional supercuspidal representation. Since we know that $\pi_L$ is a dihedral supercuspidal representation, we know it corresponds to a quadratic extension $L^\prime$ of $L$ and a character $\varkappa$ of $L^{\prime \times}.$ Starting with a representation $\pi$ of $GL(2,F),$ we will study $a(\pi_{LK}).$ Since the order in which we do base change is immaterial, we know that $a(\pi_{LK})=a((\pi_{K})_{LK}).$ Since $LK/K$is tamely ramified we can reduce our problem to the dihedral case dealt with in the last section. For our purposes this study involving conductors was sufficient.

When $F=\mathbb{Q}_2,$ there exist two cubic extensions over $F,$ a totally ramified, non-Galois extension,$L_1=\mathbb{Q}_2[\sqrt[3]{2}]$ and an unramified, Galois extension $L_2=\mathbb{Q}_2[\alpha]$ where $\alpha^3+\alpha+1=0.$ Let us dive a bit deeper into the structure of the biquadratic extension $L^\prime\cdot LK/L$ for each cubic extension $L.$ The important elements for our purpose are the discriminants of each of these extensions, which are summarized below.
\begin{lemma} When $d(L^\prime/L_1)=2,$ we have

\begin{tabular}{|c|c|c|c|c|c|}
\hline
$d(K/F)$ & $d(L_1K/L_1)$ 
& $d(L_0/L_1)$ 
& $d(L^\prime\!\cdot\!L_1K/L_1)$ 
& $d(L^\prime\!\cdot\!L_1K/L^\prime)$ 
& $d(L^\prime\!\cdot\!L_1K/L_1K)$ \\
\hline
$2$ & $4$ & $4$ & $10$ & $6$ & $2$ \\
\hline
$3$ & $7$ & $7$ & $16$ & $12$ & $2$ \\
\hline
\end{tabular}
\end{lemma}
\begin{proof}
    Since both $L_1/F$ and $L_1K/K$ are totally tamely ramified of degree $3,$ $ d(L_1/F)=d(L_1K/K)=3-1=2.$ The rest of the proposition is just an application of discriminant in towers. Let us calculate $d(L_1K/F)=d(L_1K/L_1)+2d(L_1/F)=d(L_1K/K)+3d(K/F).$ Thus, we get $d(L_1K/L_1)=3d(K/F)-2.$ Now $d(L_{0}/L_1)=\max\{d(L^\prime/L_1),d(L_1K/L_1)\}=d(L_1K/L_1)$ since the discriminants are unequal. Now \cref{lem:disc} gives  $d(L^\prime\cdot L_1K/L_1).$ Again applying the discriminant formula in towers gives all the other entries in the table.
\end{proof}
\begin{lemma} When $d(L^\prime/L_1)=4,$  we have  

\begin{tabular}{|c|c|c|c|c|c|}
\hline
$d(K/F)$ & $d(L_1K/L_1)$ 
& $d(L_0/L_1)$ 
& $d(L^\prime\!\cdot\!L_1K/L_1)$ 
& $d(L^\prime\!\cdot\!L_1K/L^\prime)$ 
& $d(L^\prime\!\cdot\!L_1K/L_1K)$ \\

\hline
$3$ & $7$ & $7$ & $18$ & $10$ & $4$ \\
\hline
\end{tabular}
\end{lemma}
\begin{proof}
    This is verbatim the previous lemma. 
\end{proof}
\begin{remark}
    When $d(K/F)=2$ we get that $d(L_1K/L_1)=4.$ Hence it is possible that $L_1K\cong L^\prime.$ When they are not isomorphic, the third extension $L_0$ could have any discriminant less than  $4.$
\end{remark}
We also have a similar lemma when $L_2/L$ is unramified.
\begin{lemma}
When $d(L^\prime/L_2)=2$ we have 

 \begin{tabular}{|c|c|c|c|c|c|}
\hline
$d(K/F)$ & $d(L_2K/L_2)$ 
& $d(L_0/L_2)$ 
& $d(L^\prime\cdot L_2K/L_2)$ 
& $d(L^\prime\cdot L_2K/L^\prime)$ 
& $d(L^\prime \cdot L_2K/L_2K)$ \\
\hline
$2$ & $2$ & $\leq 2$ & $$ & $$ & $$ \\

\hline
$3$ & $3$ & $3$ & $8$ & $4$ & $2$ \\
\hline
\end{tabular}  
\end{lemma}
\begin{theorem}
Let $\pi$ be a minimal exceptional supercuspidal representation of $GL(2,\mathbb{Q}_2).$ Let $K$ be a ramified quadratic extension of $\mathbb{Q}_2.$ Then$$a(\pi_K)=
\begin{cases}
    a(\pi) & a(\pi)<2d, \\
    2(a(\pi)-d(K/F)) &\text{otherwise.}
\end{cases}$$
\end{theorem}
\begin{proof}
    The proof is roughly the same for all the cases. We assume $L$ is the cubic extension of $F$ such that $\pi_L\cong \text{Ind}_{L^\prime}^L(\varkappa)$ is imprimitive. We do $a(\pi)=3$ and $a(\pi)=5$ and leave $a(\pi)=7$ to the reader. Let $a(\pi)=3,$ then $a(\pi_{L_1})=3a(\pi)-4=5=a(\varkappa)+d(L^\prime/L_1)$ by the tamely ramified formula. Then either one of the following holds
    \begin{enumerate} 
        \item $a(\varkappa)=3,d(L^\prime/L_1)=2;$ 
        \item $a(\varkappa)=1,d(L^\prime/L_1)=4.$
    \end{enumerate}
The second case cannot happen since  $a(\varkappa)< d(L^\prime/L_1).$ By our discussion above, $\pi_{L_1K}=\pi({L^\prime\cdot L_1K},\varkappa_{L^\prime\cdot L_1K}).$ Thus $a(\varkappa)<d(L^\prime\cdot L_1K/L^\prime)$ in both cases and thus $a(\varkappa_{L^\prime\cdot L_1K})=a(\varkappa)=3$ independently of $K.$ Thus $a(\pi_{L_1K})=5.$ Now $a(\pi_{L_1K})=3a(\pi_K)-4,$ hence we get $a(\pi_K)=3.$ This is summarized below.
\begin{center}
    \small
    \begin{tabular}{|c|c|c|c|c|}
    \hline
    $d(K/F)$ & $a(\varkappa)=a(\varkappa_{L^\prime L_1K})$
    & $d(L^\prime L_1K/L_1K)$ & $a(\pi_{L_1K})$ & $a(\pi_K)$ \\
    \hline
    $2$ & $3$ & $2$ & $5$ & $3$ \\
    \hline
    $3$ & $3$ & $2$ & $5$ & $3$ \\
    \hline
    \end{tabular}
\end{center}
Now when $a(\pi)=5,$ we get $a(\pi_{L_2})=5$ since $L_2/F$ is unramified. Then, for a given $\pi,$ one of the following holds: \begin{enumerate}
    \item $a(\varkappa)=3,d(L^\prime/L_2)=2;$
        \item $a(\varkappa)=2,d(L^\prime/L_2)=3.$
        \end{enumerate}

   Clearly $(2)$ is not possible since $a(\varkappa)>d(L^\prime/L_2).$ Thus $a(\varkappa)=3$ and $d(L^\prime /L_2)=2.$ Since $L_2K/K$ is unramified, $a(\pi_{{L_2}K})=a(\pi_K).$ When $d(K/F)=2,$  we have three possibilities. \begin{enumerate}[label=(\alph*)]
        \item $L^\prime\cong {L_2K}:$\newline Now $\pi_{L_2K}$ is a principal series representation and hence by our discussion previously we have $a(\pi_{L_2K})=2a(\varkappa)=6.$ Thus we get $a(\pi_K)=6.$
        \item $L_0$ is unramified:\newline
        In this case both $L^\prime\cdot L_2K/L_2$ and $L^\prime\cdot L_2K/L_2K$ are unramified. Thus $a(\pi_{L_2K})=2a(\varkappa_{L^\prime\cdot L_2K})=2a(\varkappa)=6.$ Thus we get $a(\pi_{K})=6.$
        \item $d(L_{0}/L_2)=2:$\newline Thus $d(L^\prime\cdot L_2K/L^\prime)=d(L^\prime\cdot L_2K/L_2K)=2.$ Thus $a(\varkappa)>d(L^\prime\cdot L_2K/L_2)$ and hence $a(\varkappa_{L^\prime\cdot L_2K})=2a(\varkappa)-d(L^\prime\cdot L_2K/L^\prime)=2*3-2=4.$
        Thus $a(\pi_{L_2K})=a(\varkappa_{L^\prime\cdot L_2K})+d(L^\prime\cdot L_2K/L_2K)=4+2=6.$ Thus $a(\pi_K)=6.$\end{enumerate}       
When $d(K/F)=3,$  we have $d(L^\prime\cdot L_2K/L^\prime)=4$ and $a(\varkappa)<d(L^\prime\cdot L_2K/L^\prime).$ Now we use \cite{MR554237} Chapter 5 Propositions 4 and 5 for the extension $L^\prime\cdot L_2K/L^\prime.$ Here $t=d(L^\prime\cdot L_2K/L^\prime)-1=3.$ By Proposition 4, $N(U_{L^\prime\cdot L_2K}^3)\subset U_{L^\prime}^3.$ By Proposition $5(ii),N(U_{L^\prime\cdot L_2K}^2) \not \subset  U_{L^\prime}^3.$Thus $a(\varkappa\circ N_{L^\prime\cdot L_2K/L^\prime})=3.$ Hence $a(\pi_K)=a(\pi_{{L_2}K})=a(\varkappa\circ N_{L^\prime\cdot L_2K/L^\prime})+d(L^\prime\cdot L_2K/L_2K)=3+2=5.$

When $a(\pi)=7,$ we get $a(\pi_{L_1})=17$ and hence for a given $\pi$ one of the following holds: \begin{enumerate} 
        \item $a(\varkappa)=15,d(L^\prime/L_1)=2$
        \item $a(\varkappa)=13,d(L^\prime/L_1)=4$
        \item $a(\varkappa)=11,d(L^\prime/L_1)=6$
        \item $a(\varkappa)=10,d(L^\prime/L_1)=7$
    \end{enumerate}
     A case-by-case analysis as above gives the result for $a(\pi_K).$ 
\end{proof}
\begin{remark}
    The case $a(\pi)=3$ of the previous result is also stated in \cite{MR721997} Proposition 3.9.
\end{remark}

Now let us study the base change of non-minimal exceptional representations.
\begin{theorem}
    Let $\pi$ be an  exceptional supercuspidal representation of $GL(2,\mathbb{Q}_2).$ Let $\chi$ be a non-trivial character of $F^\times$ such that $\pi \cong \pi_0 \otimes \chi$ where $\pi_0$ is a minimal supercuspidal representation ; note that $\pi$ itself may already be  minimal.
    \begin{enumerate}
    \item Assume that $a(\pi)\neq 2d(K/F).$Then,
    $$a(\pi_K)=\begin{cases}
        a(\pi)&a(\pi)<2d(K/F),\\
        2(a(\pi)-d(K/F))&a(\pi)>2d(K/F)
    \end{cases}$$
        \item Assume that when $d(K/F)=2,a(\pi)=4. $
     Then $$a(\pi_K)=
      3$$
    \item Let $a(\pi)=6 \text{ and } d(K/F)=3.$ Then $$a(\pi_K)=\begin{cases}
        
      5&a(\pi_0)=5,\\ 
      3& a(\pi_0)=3 \text{ and }\chi=w_{K/F} \text{ on } o^\times_F,\\
      4& a(\pi_0)=3 \text{ and }\chi\neq w_{K/F} \text{ on } o^\times_F.\\
    \end{cases}$$
\end{enumerate}
\end{theorem}
\begin{proof}
The case where $\pi$ is minimal follows from the previous theorem.
   We know that $(\pi \otimes \chi)_K=\pi_K\otimes \chi_K.$ Using \cite{MR476703} Proposition 3.4, we see that for any representation $\pi,$ $a(\pi\otimes \chi)\leq \max\{a(\pi),2a(\chi)\}$ with equality if $\pi$ is minimal or $a(\pi)\neq 2a(\chi).$ 
       Let us look at the case where $a(\pi)=4.$ Then we have $a(\pi_0)=3$ and $a(\chi)=2.$ When $d(K/F)=2,$ then $a(\chi)_K=0$ hence $a(\pi_K)=a((\pi_0)_K)=3.$ When  $d(K/F)=3, a(\chi_K)=2$ and hence $a(\pi_K)=4.$

       Now when $a(\pi)=6,$ we have either $a(\chi)=3$ and $a(\pi_0)=3 \text{ or } 5.$ When $a(\pi_0)=3$ and$\chi=w_{K/F} \text{ on } o^\times_F,$ clearly $\chi_K$ is unramified and hence $a(\pi_K)=a((\pi_0)_K)=a(\pi_0).$ When  $a(\pi_0)=3,a(\chi)=3 \text{ and }  \chi\neq w_{K/F} \text{ on } o^\times_F, \chi=\chi_2w_{K/F},$ where $\chi_2$ is a character of conductor $2.$ Thus $a((\pi_0\otimes \chi)_K)=a((\pi_0\otimes \chi_2)_K),$ which 
       is $4.$ When $a(\pi_0)=5,$$a(\chi_K)=0 \text{ or 2}.$ In either way $a((\pi_0)_K)>2a(\chi_K)$ and hence $a(\pi_K)=a((\pi_0)_K)=5.$  

       In all the other cases $a(\chi)>d(K/F).$   Now applying this to our case, we have $a((\pi_0\otimes \chi)_K)\leq \max\{2(a(\pi_0)-d),2a(\chi_K)\},$ with equality only if $2(a(\pi_0)-d)\neq 2a(\chi_K).$ But if they were equal, it would imply $a(\pi_0)=a(\chi_K)+d.$ This means that  $a(\pi_0)=2a(\chi).$ This cannot happen since $a(\pi_0)$ is odd. 
    
\end{proof}

The results of this section are summarized as follows in \cref{tab:dyadic-exceptional}. 
\begin{table}
    \centering
    \begin{tabular}{|c|c|}
    \hline
        Conditions &$a(\pi_K)$ \\\hline
        
         $a(\pi)<2d(K/F)$  &$a(\pi)$ \\ \hline
         $a(\pi)>2d(K/F)$  &$2(a(\pi)-d(K/F))$ \\ \hline
        
         $a(\pi)=4$, $d(K/F)=2$  &$3$ \\ \hline
         $\pi\cong \pi_0\otimes\chi, a(\pi)=6 , a(\pi_0)=5 , d(K/F)=3$  &$5$ \\ \hline
         $\pi\cong \pi_0\otimes\chi, a(\pi)=6, a(\pi_0)=3 ,  d(K/F)=3, \chi=w_{K/F} \text{ on }\mathfrak{o}^\times_F$  &$3$ \\ \hline
         $\pi\cong \pi_0\otimes\chi, a(\pi)=6,a(\pi_0)=3, d(K/F)=3, \chi\neq w_{K/F} \text{ on }\mathfrak{o}^\times_F$  &$4$ \\ \hline
       
    \end{tabular}
    \caption{Summary of results when $F=\mathbb{Q}_2$, $\pi$ is an exceptional representation, $K/F$ is a ramified quadratic extension}
    \label{tab:dyadic-exceptional}
\end{table}
\paragraph{We now provide an explicit description of the base change of representations with $a(\pi)=3.$ By
\cite[Thm 5.1]{MR476703}  when residue cardinality of $K$ is 2 there exists only one orbit of exceptional supercuspidal representation of $GL(2,K)$ with $a(\pi)=3.$ Furthermore, if $\pi$ is assumed to have trivial central character, then there exist precisely two such representations, which are unramified twists of each other. It can easily be seen that the local root number $\varepsilon(\pi)$ are of opposite sign for these representations(Since $\pi$ has trivial central character, the local root number does not depend on choice of the additive character.) Starting with an exceptional supercuspidal representation $\pi$ of $GL(2,\mathbb{Q}_2)$ with trivial central character and $a(\pi)=3,$ we determine $\pi_K$ for each totally ramified quadratic extension $K$ of $\mathbb{Q}_2,$ in terms of its $\varepsilon$-factors}

\begin{theorem}
Let $\pi$ be an exceptional supercuspidal representation of $GL(2,\mathbb{Q}_2)$ with trivial central character and $a(\pi)=3.$ For a ramified quadratic extension $K$ of $\mathbb{Q}_2,$ $\pi_K$ is the exceptional supercuspidal representation of $GL(2,K)$ with $a(\pi_K)=3$ and $$\varepsilon(\pi_K)=\begin{cases}
    -\varepsilon(\pi)& \text{ if } d(K/\mathbb{Q}_2)=2,\\
     \varepsilon(\pi)& \text{ if } d(K/\mathbb{Q}_2)=3.
\end{cases}$$
\begin{proof}
Let $w_{K/\mathbb{Q}_2}$ denote the non trivial character of $\mathbb{Q}_2^\times$ trivial on norms from $K^\times.$
    By \cite[Prop 6.9]{MR1007299}, taking $\tau$ to be the trivial character of $\mathbb{Q}_2^\times,$ we get
    $$\varepsilon(\pi_K)=\lambda_{K/\mathbb{Q}_2}^{-2}.\varepsilon(\pi).\varepsilon(\pi\otimes w_{K/\mathbb{Q}_2})$$ 
    By \cite{MR2234120} (30.4.3) $\lambda_{K/\mathbb{Q}_2}^{-2}=w_{K/\mathbb{Q}_2}(-1).$ Using \cite{MR721997} Lemma 3.8 we have $$\varepsilon(\pi\otimes w_{K/\mathbb{Q}_2})=\begin{cases}
        -w_{K/\mathbb{Q}_2}(-1)& d{(K/\mathbb{Q}_2)}=2,\\
        w_{K/\mathbb{Q}_2}(-1)& d{(K/\mathbb{Q}_2)}=3.\\
    \end{cases}$$
    Putting all of these together we get what we want.
\end{proof}     
\end{theorem}
\begin{remark}
We note that the assertion in \cite{MR4269428} Remark 3.7 concerning sporadic supercuspidal representations of level $2^6$ does not hold in general by the above mentioned lemma in \cite{MR721997}.

\end{remark}
\section{Global conductors}
In this section we address our original problem: how the level of a modular form change under base change? Let $f$ be a normalised Hecke eigenform of weight $\mathbf{k}$ and exact level $N$ over $F.$ There is a well-established theory of viewing $f$ as an automorphic representation $\pi_f=\bigotimes_\mathfrak{p} \pi_\mathfrak{p}$ of $GL_2(\mathbb{A}_{F})$ as described in standard texts such as \cite{MR401654}. This correspondence follows $N=\prod_\mathfrak{p} p^{c(\pi_\mathfrak{p})}.$  In \cite{MR2869056} the authors describe an algorithm to obtain the representation $\pi_p$ from $f,$ when $F=\mathbb{Q}.$ Thus, for a quadratic extension $K/F,$ if we denote by $f_K$ to be the base change of $f$ to $K,$ then $\text{level}(f_K)=\prod_\mathfrak{P} \mathfrak{P}^{c(\pi_{p_K})}.$ Thus our work in the previous sections should be enough to give the level of the base change once we explicitly know the local components. We make the following assumptions for a cleaner result. When dyadic primes are ramified in $F$ or divide $N,$ we assume that they are either (1) totally ramified, or (2) split completely in $F.$ This ensures that for any dyadic prime $\mathfrak{p}_{0}$ of $F,$ we have $F_\mathfrak{p_{0}}\cong \mathbb{Q}_2.$
\subsection*{Hypothesis}
\begin{enumerate}
    \item [\textbf{(H1)}] If $\pi_\mathfrak{p}$ is a principal series  representation, then we assume $\pi_p$ has trivial central character.
\newline The next 3 hypotheses are local conditions at dyadic primes $\mathfrak{p}_{0}.$
    \item [\textbf{(H2)}] When $\pi_{\mathfrak{p}_{0}}$ is a supercuspidal representation induced from the unramified quadratic extension of $\mathbb{Q}_2,$ $a(\pi)> 2d(K/F).$
    \item [\textbf{(H3)}] When $\pi_{\mathfrak{p}_{0}}$ is a supercuspidal representation induced from a quadratic extension $K$ of $\mathbb{Q}_2$ with discriminant $2,$ and $\text{if val}_2(d(K/\mathbb{Q}_2))=3,$ we assume $a(\pi_2)\neq 6.$
\end{enumerate}
As we remarked in the introduction, let us first deal with the case when dyadic primes are unramified in $K$ or dyadic primes do not divide $N.$ Let $K/F$ be a  quadratic extension of totally real number fields, with discriminant $D.$Let
$$\mathcal{R}=\{\mathfrak{p} \in \mathfrak{o}_F; \mathfrak{p}\mid D \text{ and } \mathfrak{p}\mid N\}$$
$$\mathcal{S}=\{\mathfrak{p} \in \mathfrak{o}_F; \mathfrak{p}\nmid D \text{ and } \mathfrak{p}\mid N\}$$
Let $\mathcal{R^\prime}\subset \mathcal{R}$ be such that $(\pi_{\mathfrak{p}})_{K_{\mathfrak{P}}}$ is not unramified for any $\mathfrak{P}$ lying over $\mathfrak{p}.$
Let us first prove Theorem 1.1.
\begin{proof}[Proof of Theorem \ref{thm:1.1}] Clearly when $v_{p}(N)=0,$ $\pi_p$ is unramified. Thus $(\pi_p)_{K_\mathfrak{P}}$ is also unramfied and hence $v_\mathfrak{P}(\mathfrak{N})=0$ for $\mathfrak{P}$ dividing $\mathfrak{p}.$ Now when 
 $\mathfrak{p}\in \mathcal{S},$by Corollary \ref{cor:unbase}, we have $v_\mathfrak{P}(\mathfrak{N})=v_\mathfrak{p}(N).$ Let $\mathfrak{p}$ be a dyadic prime. By our assumption on dyadic primes  we get that $\mathfrak{p}$does not lie in $\mathcal{R}.$ Thus, for any dyadic primes we get $v_\mathfrak{P}(\mathfrak{N})=v_\mathfrak{p}(N).$ Now if $\mathfrak{p}$ is in $\mathcal{R},$ we assume $v_{p}(N)>2$ or $\pi_p$ is supercuspidal. In either of these cases, we get $v_{\mathfrak{P}}(\mathfrak{N})=2(v_{p}(N)-1)$ for any $\mathfrak{P}$ dividing p, by Theorem \ref{thm:tame}.

\end{proof}
Now let us continue to the case where dyadic primes divide $N$ and are ramified in $K.$ This is a generalisation of Theorem \ref{thm:1.2}.

 Let $\mathfrak{p}$ be a prime ideal in $\mathfrak{o}_F$,$\mathfrak{P}$ be the prime ideal lying above $\mathfrak{p}.$  Let $\mathcal{R^\prime}\subset \mathcal{R}$ be such that $(\pi_{\mathfrak{p}})_{K_{\mathfrak{P}}}$ is not unramified for any $\mathfrak{P}$ lying over $\mathfrak{p}.$ Set $\mathfrak{d}_\mathfrak{p}=v_{\mathfrak{p}}(D).$ Assume that all the hypotheses above are satisfied. 
\begin{theorem}
   Let  $f, f_K,N,\mathfrak{N}$ be all as above. Then
   $$\mathfrak{N}=\prod_{\substack{\mathfrak{P}\mid \mathfrak{p}\\ \mathfrak{p}\in \mathfrak{R}^\prime}}\mathfrak{P}^{c(f,\mathfrak{p})}.\prod_{
\substack{\mathfrak{P}\mid \mathfrak{p}\\ \mathfrak{p}\in \mathcal{S}}}\mathfrak{P}^{v_\mathfrak{p}(N)}$$ where $c(f,\mathfrak{p})$ is an integer given as follows.
   \begin{enumerate}
       \item If $\mathfrak{p}$ is odd and $\pi_\mathfrak{p}$ is not a special representation $$c(f,\mathfrak{p})=2(v_{\mathfrak{p}}(N)-1).$$
       \item If $\mathfrak{p}$ is odd and $\pi_{\mathfrak{p}}$ is a special representation then $$c(f,\mathfrak{p})=\begin{cases}
           1& \text{ if }v_{\mathfrak{p}}(N)=1\\ 1& \text{ if }v_{\mathfrak{p}}(N)=2 \text{ and } \pi_\mathfrak{p} \text{ has trivial central character},\\
         2(v_\mathfrak{p}(N)-1)& \text{ otherwise.}
       \end{cases}$$
       For the following assume that $\mathfrak{p}$ is dyadic.
       \item If $v_\mathfrak{p}(N)<2\mathfrak{d_p}$ then
       $$c(f,\mathfrak{p})=v_{\mathfrak{p}}(N).$$

       \item If $v_\mathfrak{p}(N)>2\mathfrak{d_p}$ then
       $$c(f,\mathfrak{p})=2(v_{\mathfrak{p}}(N)-\mathfrak{d_p}).$$
       \newline
       For the remaining let us assume that $v_\mathfrak{p}(N)=2\mathfrak{d_p}.$\newline
       \item If  $\pi_{\mathfrak{p}}$ is a principal series representation $B(\chi,\chi^{-1})$ or a special representation $\text{St}\otimes \chi$  then $$c(f,\mathfrak{p})=\begin{cases}
           0& \pi_{\mathfrak{p}}\text{ is principal series and }\chi=w_{K_\mathfrak{P}/F_\mathfrak{p}} \text{ on }\mathfrak{o}^\times_{F_{\mathfrak{p}}} ,\\
            1& \pi_{\mathfrak{p}}\text{ is special representation and }\chi=w_{K_\mathfrak{P}/F_\mathfrak{p}} \text{ on }\mathfrak{o}^\times_{F_{\mathfrak{p}}} ,\\
           v_{\mathfrak{p}}(N)-2& \text{ otherwise.}\\
       \end{cases}$$
       
       \item If $\pi_\mathfrak{p}=\pi(L,\varkappa)$ is a dihedral supercuspidal representation such that 
       \begin{enumerate}
        \item $L$ is not unramified
        \item When $\mathfrak{d_p}=3, d(L/F_\mathfrak{p})\neq 2.$
       \end{enumerate} Then,$$c(f,\mathfrak{p})=v_\mathfrak{p}(N)=2
       (v_\mathfrak{p}(N)-\mathfrak{d_p}).$$
             \item If  $\pi_\mathfrak{p}$ is exceptional and ${\mathfrak{d_p}}=2.$Then, $$c(f,\mathfrak{p})=
           3$$
                
       \item If $\pi_\mathfrak{p}$ is exceptional and $\mathfrak{d_p}=3$ then we have $\pi_\mathfrak{p}=\pi_0\otimes \chi$ with $a(\chi)=3.$ Then $$c(f,\mathfrak{p})=\begin{cases}
       3&\text{if }a(\pi_0)=3 \text{ and } \chi=w_{K_\mathfrak{P}/F_\mathfrak{p}} \text{ on }\mathfrak{o}^\times_{F_{\mathfrak{p}}},\\
       4& \text{if }a(\pi_0)=3 \text{ and } \chi\neq w_{K_\mathfrak{P}/F_\mathfrak{p}} \text{ on }\mathfrak{o}^\times_{F_{\mathfrak{p}}},\\ 
       5& \text{if }a(\pi_0)=5.\\ 
   \end{cases}$$
    \end{enumerate}
  \end{theorem} 
\begin{remark}
    If hypothesis \textbf{(H3)} is not true,  then $c(f,\mathfrak{p})=4, 6\text{ or }7.$ See Remark \ref{rmk:conductor}.
\end{remark}
Given a modular form $f$ we have determined what the level of the base change $f_K$ can be. Now let us try to answer the reverse. Let $K/F$ be an extension of totally real fields of discriminant $D$, both of them totally split or ramified at the dyadic places. Set $\mathfrak{P}$ be a prime ideal in $\mathfrak{o}_K$ and $\mathfrak{p}=\mathfrak{P}\cap \mathfrak{o}_F.$ Let $\mathfrak{d}_\mathfrak{P}=v_{\mathfrak{p}}(D).$ Let $\mathfrak{N}=\prod_{i}\mathfrak{P}^{v_{\mathfrak{P}_i}(\mathfrak{N})}_i$ be an ideal in $\mathfrak{o}_K,$ and let $\mathfrak{F}$ be a newform of level $\mathfrak{N}.$ We determine whether there exists a modular form $f$ over $F$ of level $\mathfrak{n},$ satisfying hypotheses $\textbf{(H1)}$ through $\textbf{(H3)},$ such that $\mathfrak{F}$ is a base change lift of $f.$ Let $\mathcal{R}$ be the set of primes of $F$ ramified in $K.$ 
\begin{enumerate}
 \item Since $f_K=(f\otimes w_{K/F})_K,$ for $w_{K/F}$ the quadratic character associated to the extension via class field theory, the level of $f$ is not unique in general.
 \item If $\mathfrak{N}$ is the level of a modular form that is a base change, then $v_{\mathfrak{P}}(\mathfrak{N})=v_{\sigma(\mathfrak{P})}(\mathfrak{N}),$ for $\sigma$ the non-trivial element in $\text{Gal}(K/F).$
\end{enumerate}
The first part above says that the level $\mathfrak{n}$ is not unique in general, while the second part is a necessary condition for an ideal to be the level of a base change.
\begin{theorem}
    Let $\mathfrak{N}=\prod_{i}\mathfrak{P}_{i}^{v_{\mathfrak{P}_{i}}(\mathfrak{N})}$ be an ideal in $\mathfrak{o}_K$ such that $(2)$ above is true. Then there exists a modular form of level $\mathfrak{N}$ such that it is a base change of a modular form satisfying $\textbf{(H1)}$ through $\textbf{(H3)}$ only if the following hold for $\mathfrak P \mid \mathfrak p \in \mathcal R:$\begin{enumerate}
    
        \item $v_{\mathfrak{P}}(\mathfrak{N})$ is even when $v_{\mathfrak{P}}(\mathfrak{N})>2\mathfrak{d_P}.$
        \item When $\mathfrak{d_P}=3,v_{\mathfrak{P}}(\mathfrak{N})\neq 2.$ 
    \end{enumerate}
    Moreover the possible level $\mathfrak{n}$ of a modular form $f$ over $F$ that gives a modular form of level $\mathfrak{N}$ over $K$ is given as follows
    \begin{enumerate}
        \item If $\mathfrak{P}\mid \mathfrak{p}\notin \mathcal{R},$ then $v_{\mathfrak{P}}(\mathfrak{N})=v_{\mathfrak{p}}(\mathfrak{n}).$ \newline
        For the following assume that $\mathfrak{P}\mid \mathfrak{p}\in \mathcal{R}.$
        \item If $v_{\mathfrak{P}}(\mathfrak{N})=0 \text{ or } 1$ then $v_{\mathfrak{p}}(\mathfrak{n})=v_{\mathfrak{P}}(\mathfrak{N}) \text{ or } 2\mathfrak{d_P}.$

        \item If $v_{\mathfrak{P}}(\mathfrak{N})>2\mathfrak{d_P},$ we have $v_{\mathfrak{p}}(\mathfrak{n})=\dfrac{v_{\mathfrak{P}}(\mathfrak{N})}{2}+\mathfrak{d_P}.$\newline
        For the following assume that $\mathfrak{d_P}=2.$  
        \item If $v_{\mathfrak{P}}(\mathfrak{N})=4,$ then $v_{\mathfrak{p}}(\mathfrak{n})=2.$ 
        \item If $v_{\mathfrak{P}}(\mathfrak{N})=3,$ then $v_{\mathfrak{p}}(\mathfrak{n})=3 \text{ or } 4.$ \newline
        For the following assume that $\mathfrak{d_P}=3.$
        \item If $v_{\mathfrak{P}}(\mathfrak{N})=3,$ then  $v_{\mathfrak{p}}(\mathfrak{n})=3 \text{ or } 6.$ 
        \item If $v_{\mathfrak{P}}(\mathfrak{N})=4,$ then $v_{\mathfrak{p}}(\mathfrak{n})=4,5 \text{ or } 6.$ 
        \item If $v_{\mathfrak{P}}(\mathfrak{N})=5,$ then $v_{\mathfrak{p}}(\mathfrak{n})=5\text{ or }6.$ 
        \item If $v_{\mathfrak{P}}(\mathfrak{N})=6,$ then
 $v_{\mathfrak{p}}(\mathfrak{n})=3,4 \text{ or } 6.$         
    \end{enumerate}
    \end{theorem}
\printbibliography

@book {MR554237,
    AUTHOR = {Serre, Jean-Pierre},
     TITLE = {Local fields},
    SERIES = {Graduate Texts in Mathematics},
    VOLUME = {67},
      NOTE = {Translated from the French by Marvin Jay Greenberg},
 PUBLISHER = {Springer-Verlag, New York-Berlin},
      YEAR = {1979},
     PAGES = {viii+241},
      ISBN = {0-387-90424-7},
   MRCLASS = {12Bxx},
  MRNUMBER = {554237},
}

@book {MR1007299,
    AUTHOR = {Arthur, James and Clozel, Laurent},
     TITLE = {Simple algebras, base change, and the advanced theory of the
              trace formula},
    SERIES = {Annals of Mathematics Studies},
    VOLUME = {120},
 PUBLISHER = {Princeton University Press, Princeton, NJ},
      YEAR = {1989},
     PAGES = {xiv+230},
      ISBN = {0-691-08517-X},
   MRCLASS = {22E55 (11F70 11F72 11R39)},
  MRNUMBER = {1007299},
MRREVIEWER = {Stephen\ Gelbart},
}

@book {MR2234120,
    AUTHOR = {Bushnell, Colin J. and Henniart, Guy},
     TITLE = {The local {L}anglands conjecture for {$\rm GL(2)$}},
    SERIES = {Grundlehren der mathematischen Wissenschaften [Fundamental
              Principles of Mathematical Sciences]},
    VOLUME = {335},
 PUBLISHER = {Springer-Verlag, Berlin},
      YEAR = {2006},
     PAGES = {xii+347},
      ISBN = {978-3-540-31486-8},
   MRCLASS = {22E50 (11-02 11S37 22-02)},
  MRNUMBER = {2234120},
MRREVIEWER = {Alexandru\ Ioan\ Badulescu},
       DOI = {10.1007/3-540-31511-X},
       URL = {https://doi.org/10.1007/3-540-31511-X},
}

@book {MR427267,
    AUTHOR = {Weil, Andr\'e},
     TITLE = {Basic number theory},
    SERIES = {Die Grundlehren der mathematischen Wissenschaften},
    VOLUME = {Band 144},
   EDITION = {Third},
 PUBLISHER = {Springer-Verlag, New York-Berlin},
      YEAR = {1974},
     PAGES = {xviii+325},
   MRCLASS = {12-02},
  MRNUMBER = {427267},
}

@book {MR1697859,
    AUTHOR = {Neukirch, J\"urgen},
     TITLE = {Algebraic number theory},
    SERIES = {Grundlehren der mathematischen Wissenschaften [Fundamental
              Principles of Mathematical Sciences]},
    VOLUME = {322},
      NOTE = {Translated from the 1992 German original and with a note by
              Norbert Schappacher,
              With a foreword by G. Harder},
 PUBLISHER = {Springer-Verlag, Berlin},
      YEAR = {1999},
     PAGES = {xviii+571},
      ISBN = {3-540-65399-6},
   MRCLASS = {11Rxx (11-02 11S15 11S31 14C40)},
  MRNUMBER = {1697859},
MRREVIEWER = {Cornelius\ Greither},
       DOI = {10.1007/978-3-662-03983-0},
       URL = {https://doi.org/10.1007/978-3-662-03983-0},
}

@book {MR1915966,
    AUTHOR = {Fesenko, I. B. and Vostokov, S. V.},
     TITLE = {Local fields and their extensions},
    SERIES = {Translations of Mathematical Monographs},
    VOLUME = {121},
   EDITION = {Second},
      NOTE = {With a foreword by I. R. Shafarevich},
 PUBLISHER = {American Mathematical Society, Providence, RI},
      YEAR = {2002},
     PAGES = {xii+345},
      ISBN = {0-8218-3259-X},
   MRCLASS = {11Sxx (11S31)},
  MRNUMBER = {1915966},
MRREVIEWER = {Jerzy\ Browkin},
       DOI = {10.1090/mmono/121},
       URL = {https://doi-org.ezproxy.lib.ou.edu/10.1090/mmono/121},
}

@article {MR4705658,
    AUTHOR = {\v{C}esnavi\v{c}ius, K\k{e}stutis and Neururer, Michael and
              Saha, Abhishek},
     TITLE = {The {M}anin constant and the modular degree},
   JOURNAL = {J. Eur. Math. Soc. (JEMS)},
  FJOURNAL = {Journal of the European Mathematical Society (JEMS)},
    VOLUME = {26},
      YEAR = {2024},
    NUMBER = {2},
     PAGES = {573--637},
      ISSN = {1435-9855,1435-9863},
   MRCLASS = {11G05 (11F70 11F80 11F85 11G18 11L05)},
  MRNUMBER = {4705658},
MRREVIEWER = {John\ T.\ Cullinan},
       DOI = {10.4171/jems/1367},
       URL = {https://doi.org/10.4171/jems/1367},
}

@article {MR476703,
    AUTHOR = {Tunnell, Jerrold B.},
     TITLE = {On the local {L}anglands conjecture for {$GL(2)$}},
   JOURNAL = {Invent. Math.},
  FJOURNAL = {Inventiones Mathematicae},
    VOLUME = {46},
      YEAR = {1978},
    NUMBER = {2},
     PAGES = {179--200},
      ISSN = {0020-9910,1432-1297},
   MRCLASS = {12B25 (10D15 22E50)},
  MRNUMBER = {476703},
MRREVIEWER = {K.\ Shiratani},
       DOI = {10.1007/BF01393255},
       URL = {https://doi.org/10.1007/BF01393255},
}

@article {MR253990,
    AUTHOR = {Doi, Koji and Naganuma, Hidehisa},
     TITLE = {On the functional equation of certain {D}irichlet series},
   JOURNAL = {Invent. Math.},
  FJOURNAL = {Inventiones Mathematicae},
    VOLUME = {9},
    year   = {1969/70},
     PAGES = {1--14},
      ISSN = {0020-9910,1432-1297},
   MRCLASS = {10.22},
  MRNUMBER = {253990},
MRREVIEWER = {O.\ H.\ K\"orner},
       DOI = {10.1007/BF01389886},
       URL = {https://doi.org/10.1007/BF01389886},
}

@book {MR406936,
    AUTHOR = {Saito, Hiroshi},
     TITLE = {Automorphic forms and algebraic extensions of number fields},
      NOTE = {Department of Mathematics, Kyoto University, Lectures in
              Mathematics, No. 8},
 PUBLISHER = {Kinokuniya Book Store Co., Ltd., Tokyo},
      YEAR = {1975},
     PAGES = {iv+183},
   MRCLASS = {10D20 (14D20)},
  MRNUMBER = {406936},
MRREVIEWER = {Takuro\ Shintani},
}

@article {MR337789,
    AUTHOR = {Casselman, William},
     TITLE = {On some results of {A}tkin and {L}ehner},
   JOURNAL = {Math. Ann.},
  FJOURNAL = {Mathematische Annalen},
    VOLUME = {201},
      YEAR = {1973},
     PAGES = {301--314},
      ISSN = {0025-5831,1432-1807},
   MRCLASS = {10D15 (22E50)},
  MRNUMBER = {337789},
MRREVIEWER = {T.\ Miyake},
       DOI = {10.1007/BF01428197},
       URL = {https://doi.org/10.1007/BF01428197},
}

@article {MR620708,
    AUTHOR = {Jacquet, H. and Piatetski-Shapiro, I. I. and Shalika, J.},
     TITLE = {Conducteur des repr\'esentations du groupe lin\'eaire},
   JOURNAL = {Math. Ann.},
  FJOURNAL = {Mathematische Annalen},
    VOLUME = {256},
      YEAR = {1981},
    NUMBER = {2},
     PAGES = {199--214},
      ISSN = {0025-5831,1432-1807},
   MRCLASS = {22E50 (10D40 12A67)},
  MRNUMBER = {620708},
MRREVIEWER = {J.\ Tunnell},
       DOI = {10.1007/BF01450798},
       URL = {https://doi.org/10.1007/BF01450798},
}

@book {MR574808,
    AUTHOR = {Langlands, Robert P.},
     TITLE = {Base change for {${\rm GL}(2)$}},
    SERIES = {Annals of Mathematics Studies},
    VOLUME = {No. 96},
 PUBLISHER = {Princeton University Press, Princeton, NJ; University of Tokyo
              Press, Tokyo},
      YEAR = {1980},
     PAGES = {vii+237},
      ISBN = {0-691-08263-4},
   MRCLASS = {10D40 (10-02 12A67 22E55)},
  MRNUMBER = {574808},
MRREVIEWER = {Stephen\ Gelbart},
}

@incollection {MR546613,
    AUTHOR = {G\'erardin, P. and Labesse, J.-P.},
     TITLE = {The solution of a base change problem for {${\rm GL}(2)$}\
              (following {L}anglands, {S}aito, {S}hintani)},
 BOOKTITLE = {Automorphic forms, representations and {$L$}-functions
              ({P}roc. {S}ympos. {P}ure {M}ath., {O}regon {S}tate {U}niv.,
              {C}orvallis, {O}re., 1977), {P}art 2},
    SERIES = {Proc. Sympos. Pure Math.},
    VOLUME = {XXXIII},
     PAGES = {115--133},
 PUBLISHER = {Amer. Math. Soc., Providence, RI},
      YEAR = {1979},
      ISBN = {0-8218-1437-0},
   MRCLASS = {10D40 (22E55)},
  MRNUMBER = {546613},
MRREVIEWER = {Stephen\ Gelbart},
}

@article {MR1685898,
    AUTHOR = {Bushnell, Colin J. and Henniart, Guy},
     TITLE = {Local tame lifting for {${\rm GL}(n)$}. {II}. {W}ildly
              ramified supercuspidals},
   JOURNAL = {Ast\'erisque},
  FJOURNAL = {Ast\'erisque},
    NUMBER = {254},
      YEAR = {1999},
     PAGES = {vi+105},
      ISSN = {0303-1179,2492-5926},
   MRCLASS = {11S37 (11F70 22E50)},
  MRNUMBER = {1685898},
MRREVIEWER = {Marko\ Tadi\'c},
}

@article {MR2130587,
    AUTHOR = {Bushnell, Colin J. and Henniart, Guy},
     TITLE = {Local tame lifting for {${\rm GL}(n)$}. {III}. {E}xplicit base
              change and {J}acquet-{L}anglands correspondence},
   JOURNAL = {J. Reine Angew. Math.},
  FJOURNAL = {Journal f\"ur die Reine und Angewandte Mathematik. [Crelle's
              Journal]},
    VOLUME = {580},
      YEAR = {2005},
     PAGES = {39--100},
      ISSN = {0075-4102,1435-5345},
   MRCLASS = {11S37 (11F70 22E50)},
  MRNUMBER = {2130587},
MRREVIEWER = {Marko\ Tadi\'c},
       DOI = {10.1515/crll.2005.2005.580.39},
       URL = {https://doi.org/10.1515/crll.2005.2005.580.39},
}

@incollection {MR546611,
    AUTHOR = {Shintani, Takuro},
     TITLE = {On liftings of holomorphic cusp forms},
 BOOKTITLE = {Automorphic forms, representations and {$L$}-functions
              ({P}roc. {S}ympos. {P}ure {M}ath., {O}regon {S}tate {U}niv.,
              {C}orvallis, {O}re., 1977), {P}art 2},
    SERIES = {Proc. Sympos. Pure Math.},
    VOLUME = {XXXIII},
     PAGES = {97--110},
 PUBLISHER = {Amer. Math. Soc., Providence, RI},
      YEAR = {1979},
      ISBN = {0-8218-1437-0},
   MRCLASS = {10D40 (22E50)},
  MRNUMBER = {546611},
MRREVIEWER = {Stephen\ Gelbart},
}

@article {MR4522693,
    AUTHOR = {Atobe, Hiraku and Kondo, Satoshi and Yasuda, Seidai},
     TITLE = {Local newforms for the general linear groups over a
              non-archimedean local field},
   JOURNAL = {Forum Math. Pi},
  FJOURNAL = {Forum of Mathematics. Pi},
    VOLUME = {10},
      YEAR = {2022},
     PAGES = {Paper No. e24, 56},
      ISSN = {2050-5086},
   MRCLASS = {11F70 (22E50)},
  MRNUMBER = {4522693},
MRREVIEWER = {Solomon\ Friedberg},
       DOI = {10.1017/fmp.2022.17},
       URL = {https://doi.org/10.1017/fmp.2022.17},
}

@book {MR401654,
    AUTHOR = {Jacquet, H. and Langlands, R. P.},
     TITLE = {Automorphic forms on {${\rm GL}(2)$}},
    SERIES = {Lecture Notes in Mathematics},
    VOLUME = {Vol. 114},
 PUBLISHER = {Springer-Verlag, Berlin-New York},
      YEAR = {1970},
     PAGES = {vii+548},
   MRCLASS = {10D15 (12A65 12A70 22E55)},
  MRNUMBER = {401654},
MRREVIEWER = {Stephen\ Gelbart},
}

@article {MR2869056,
    AUTHOR = {Loeffler, David and Weinstein, Jared},
     TITLE = {On the computation of local components of a newform},
   JOURNAL = {Math. Comp.},
  FJOURNAL = {Mathematics of Computation},
    VOLUME = {81},
      YEAR = {2012},
    NUMBER = {278},
     PAGES = {1179--1200},
      ISSN = {0025-5718,1088-6842},
   MRCLASS = {11F70 (11F11)},
  MRNUMBER = {2869056},
MRREVIEWER = {Nathan\ C.\ Ryan},
       DOI = {10.1090/S0025-5718-2011-02530-5},
       URL = {https://doi.org/10.1090/S0025-5718-2011-02530-5},
}

@article {MR4356848,
    AUTHOR = {Roy, Manami},
     TITLE = {Paramodular forms coming from elliptic curves},
   JOURNAL = {J. Number Theory},
  FJOURNAL = {Journal of Number Theory},
    VOLUME = {233},
      YEAR = {2022},
     PAGES = {126--157},
      ISSN = {0022-314X,1096-1658},
   MRCLASS = {11F46 (11F70 11G07 14H52)},
  MRNUMBER = {4356848},
MRREVIEWER = {F.\ Cl\'ery},
       DOI = {10.1016/j.jnt.2021.06.007},
       URL = {https://doi.org/10.1016/j.jnt.2021.06.007},
}

@article {MR4878181,
    AUTHOR = {Banerjee, Debargha and Mandal, Tathagata and Mondal, Sudipa},
     TITLE = {Two properties of symmetric cube transfers of modular forms},
   JOURNAL = {J. Number Theory},
  FJOURNAL = {Journal of Number Theory},
    VOLUME = {275},
      YEAR = {2025},
     PAGES = {160--195},
      ISSN = {0022-314X,1096-1658},
   MRCLASS = {11F70 (11F80)},
  MRNUMBER = {4878181},
MRREVIEWER = {Oliver\ Stein},
       DOI = {10.1016/j.jnt.2024.12.013},
       URL = {https://doi.org/10.1016/j.jnt.2024.12.013},
}

@article {MR721997,
    AUTHOR = {Tunnell, Jerrold B.},
     TITLE = {Local {$\epsilon $}-factors and characters of {${\rm GL}(2)$}},
   JOURNAL = {Amer. J. Math.},
  FJOURNAL = {American Journal of Mathematics},
    VOLUME = {105},
      YEAR = {1983},
    NUMBER = {6},
     PAGES = {1277--1307},
      ISSN = {0002-9327,1080-6377},
   MRCLASS = {22E35 (11S37 20G05)},
  MRNUMBER = {721997},
MRREVIEWER = {Joe\ Repka},
       DOI = {10.2307/2374441},
       URL = {https://doi.org/10.2307/2374441},
}

@article {MR466025,
    AUTHOR = {Kudla, Stephen S.},
     TITLE = {Theta-functions and {H}ilbert modular forms},
   JOURNAL = {Nagoya Math. J.},
  FJOURNAL = {Nagoya Mathematical Journal},
    VOLUME = {69},
      YEAR = {1978},
     PAGES = {97--106},
      ISSN = {0027-7630,2152-6842},
   MRCLASS = {10D20 (10C15)},
  MRNUMBER = {466025},
MRREVIEWER = {Hiroshi\ Saito},
       URL = {http://projecteuclid.org/euclid.nmj/1118796620},
}

@article {MR4269428,
    AUTHOR = {Dieulefait, Luis Victor and Pacetti, Ariel and Tsaknias,
              Panagiotis},
     TITLE = {On the number of {G}alois orbits of newforms},
   JOURNAL = {J. Eur. Math. Soc. (JEMS)},
  FJOURNAL = {Journal of the European Mathematical Society (JEMS)},
    VOLUME = {23},
      YEAR = {2021},
    NUMBER = {8},
     PAGES = {2833--2860},
      ISSN = {1435-9855,1435-9863},
   MRCLASS = {11F03 (11F11)},
  MRNUMBER = {4269428},
MRREVIEWER = {Spencer\ Hamblen},
       DOI = {10.4171/jems/1073},
       URL = {https://doi.org/10.4171/jems/1073},
}

@article {MR3911789,
    AUTHOR = {Kumar, Balesh and Manickam, Murugesan},
     TITLE = {On {D}oi-{N}aganuma and {S}himura liftings},
   JOURNAL = {Ramanujan J.},
  FJOURNAL = {Ramanujan Journal. An International Journal Devoted to the
              Areas of Mathematics Influenced by Ramanujan},
    VOLUME = {48},
      YEAR = {2019},
    NUMBER = {2},
     PAGES = {279--303},
      ISSN = {1382-4090,1572-9303},
   MRCLASS = {11F11 (11F32 11F37 11F41)},
  MRNUMBER = {3911789},
MRREVIEWER = {Moni\ Kumari},
       DOI = {10.1007/s11139-017-9958-6},
       URL = {https://doi.org/10.1007/s11139-017-9958-6},
}

@article {MR734781,
    AUTHOR = {Kutzko, P. C.},
     TITLE = {The exceptional representations of {${\rm Gl}\sb{2}$}},
   JOURNAL = {Compositio Math.},
  FJOURNAL = {Compositio Mathematica},
    VOLUME = {51},
      YEAR = {1984},
    NUMBER = {1},
     PAGES = {3--14},
      ISSN = {0010-437X,1570-5846},
   MRCLASS = {11S37 (22E55)},
  MRNUMBER = {734781},
MRREVIEWER = {Martin\ L.\ Karel},
       URL = {http://www.numdam.org/item?id=CM_1984__51_1_3_0},
}

@book {MR506171,
    AUTHOR = {Buhler, Joe P.},
     TITLE = {Icosahedral {G}alois representations},
    SERIES = {Lecture Notes in Mathematics},
    VOLUME = {Vol. 654},
 PUBLISHER = {Springer-Verlag, Berlin-New York},
      YEAR = {1978},
     PAGES = {ii+143},
      ISBN = {3-540-08844-X},
   MRCLASS = {12A70 (10D05)},
  MRNUMBER = {506171},
MRREVIEWER = {Ted\ Chinburg},
}
\end{document}